\documentclass{article}[12pt]
\usepackage{tikz}
\usetikzlibrary {intersections}
\usetikzlibrary {intersections,through}
\usetikzlibrary{positioning} %为了实现相对位置的设定
\usepackage{xcolor} %为了实现不同的颜色
\usepackage{mathrsfs}
\usepackage{mathtools}
\usepackage{graphicx}
\usepackage{epstopdf}
\usepackage{stmaryrd}
 \usepackage{subfigure}
\usepackage{float}
\usepackage{amsmath}
    \usepackage{bm}
\usepackage{amsfonts,amssymb}
  \usepackage{dsfont}
\usepackage{amsfonts}
\usepackage{mathrsfs}
  \usepackage{pifont}
\numberwithin{equation}{section}
 \usepackage{amsthm}
 \newtheorem{theorem}{Theorem}[section]
\newtheorem{lemma}[theorem]{Lemma}

\newtheorem{remark}[theorem]{Remark}
\newtheorem{assumption}[theorem]{Assumption}
 \usepackage{amsmath}
 \usepackage{multirow}

\begin{document}
\title{An immersed Raviart-Thomas mixed finite element method for elliptic interface problems on unfitted meshes}
\author{
Haifeng Ji\footnotemark[1]
%\qquad
%Feng Wang\footnotemark[2] \qquad
%Jinru Chen\footnotemark[3] \quad
%Zhilin Li\footnotemark[4] 
}
\footnotetext[1]{School of Science, Nanjing University of Posts and Telecommunications, Nanjing, Jiangsu 210023, China  (hfji@njupt.edu.cn)}
%\footnotetext[2]{Jiangsu Key Laboratory for NSLSCS, School of Mathematical Sciences, Nanjing Normal University, Nanjing 210023, China  (fwang@njnu.edu.cn)}
%\footnotetext[3]{Jiangsu Key Laboratory for NSLSCS, School of Mathematical Sciences, Nanjing Normal University, Nanjing 210023, China  (jrchen@njnu.edu.cn)}
%\footnotetext[4]{Department of Mathematics, North Carolina State University, Raleigh, NC 27695, USA  (zhilin@math.ncsu.edu)}
%\renewcommand{\thefootnote}{\fnsymbol{footnote}}

\date{}
\maketitle

\begin{abstract}
This paper presents a lowest-order immersed Raviart-Thomas mixed triangular finite element method for solving elliptic interface problems on unfitted meshes independent of the interface.  In order to achieve the optimal convergence rates on unfitted meshes, an immersed finite element finite (IFE) is constructed by modifying the traditional Raviart-Thomas element. Some important properties are derived including the unisolvence of IFE basis functions, the optimal approximation capabilities of the IFE space and the corresponding commuting digram. Optimal error estimates are rigorously proved for the mixed IFE method and some numerical examples are also provided to validate the theoretical analysis.
\end{abstract}

\textbf{keyword}: interface problem, mixed finite element, immersed finite element, unfitted mesh


\section{Introduction}
Let $\Omega\subset\mathbb{R}^2$  be a convex polygonal domain and $\Gamma$  be a $C^2$-smooth interface immersed in $\Omega$.  Without loss of generality, we assume that  $\Gamma$  divides $\Omega$ into two disjoint sub-domains $\Omega^+$ and $\Omega^-$ such that $\Gamma=\partial\Omega^-$, see Figure~\ref{interfacepb} for an illustration. 
We consider  the following second-order elliptic interface problem
\begin{align}
-\nabla\cdot(\widetilde{\beta}(\boldsymbol{x})\nabla u(\boldsymbol{x}))&=f(\boldsymbol{x})  \qquad\mbox{in } \Omega\backslash\Gamma,\label{p1.1}\\
[u]_{\Gamma}(\boldsymbol{x})&=0~~~~ \qquad\mbox{on } \Gamma,\label{p1.2}\\
[\widetilde{\beta}\nabla u\cdot \textbf{n}]_{\Gamma}(\boldsymbol{x})&=0~~~~ \qquad\mbox{on } \Gamma,\label{p1.3}\\
u(\boldsymbol{x})&=0 ~~~~\qquad\mbox{on } \partial\Omega,\label{p1.4}
\end{align}
where  $\textbf{n}(\boldsymbol{x})$ is the unit normal vector of the interface $\Gamma$ at point $x\in\Gamma$ pointing toward $\Omega^+$, and the notation $[v]_\Gamma$  is defined as
\begin{equation*}
[v]_{\Gamma}(\boldsymbol{x}):= v^+|_\Gamma-v^-|_\Gamma\quad \mbox{ with } ~v^s=v|_{\Omega^s}, ~s=+,-.
\end{equation*}
The coefficient $\beta(\boldsymbol{x})$  can be  discontinuous across the interface $\Gamma$ and is assumed to be piecewise smooth
\begin{align}\label{p1.5}
\widetilde{\beta}(\boldsymbol{x})=\widetilde{\beta}^+(\boldsymbol{x}) ~~\mbox{ if } x\in\Omega^+~~~ \mbox{ and }~~~ \widetilde{\beta}(\boldsymbol{x})=\widetilde{\beta}^-(\boldsymbol{x}) ~~\mbox{ if } x\in\Omega^-,
\end{align}
with $\widetilde{\beta}^s(\boldsymbol{x})\in C^1(\overline{\Omega^s})$, $s=+,-$. We also assume that there exist two positive constants $\widetilde{\beta}_{min}$ and $\widetilde{\beta}_{max}$  such that $0<\widetilde{\beta}_{min}\leq \widetilde{\beta}^s(\boldsymbol{x})\leq\widetilde{\beta}_{max}$ for all $x\in \overline{\Omega^s}$, $s=+,-$.

%\begin{figure} [htbp]
%\centering
%\includegraphics[height=4cm,clip]{interfacepb.eps}
% \caption{A diagram of the geometries of an interface problem.}\label{interfacepb} %% label for entire figure
%\end{figure}

\begin{figure} [htbp]
\centering
\begin{tikzpicture}[scale=2]
\draw  (-1,-1)--(1,-1);
\draw (-1,1)--(1,1);
\draw (-1,-1)--(-1,1);
\draw (1,-1)--(1,1);
\draw[thick][rotate=120] (0,0) ellipse [x radius=0.4, y radius=0.65];
\node at (0,0) {$\Omega^-$};
\node at (-0.75,0.75) {$\Omega^+$};
\node at (0.5,0.55) {$\Gamma$};
\end{tikzpicture}
\qquad\quad
\begin{tikzpicture}[scale=2]
\draw  (-1,-1)--(1,-1);
\draw (-1,-0.5)--(1,-0.5);
\draw (-1,0)--(1,0);
\draw (-1,0.5)--(1,0.5);
\draw (-1,1)--(1,1);
\draw (-1,-1)--(-1,1);
\draw (-0.5,-1)--(-0.5,1);
\draw (0,-1)--(0,1);
\draw (0.5,-1)--(0.5,1);
\draw (1,-1)--(1,1);
\draw (-1,1)--(1,-1);
\draw (-0.5,1)--(1,-0.5);
\draw (0,1)--(1,0);
\draw (0.5,1)--(1,0.5);
\draw (-1,0.5)--(0.5,-1);
\draw (-1,0)--(0,-1);
\draw (-1,-0.5)--(-0.5,-1);
\draw[thick][rotate=120] (0,0) ellipse [x radius=0.4, y radius=0.65];
\end{tikzpicture}
 \caption{Left: a diagram of the geometries of an interface problem; Right: an unfitted mesh}\label{interfacepb} 
\end{figure}
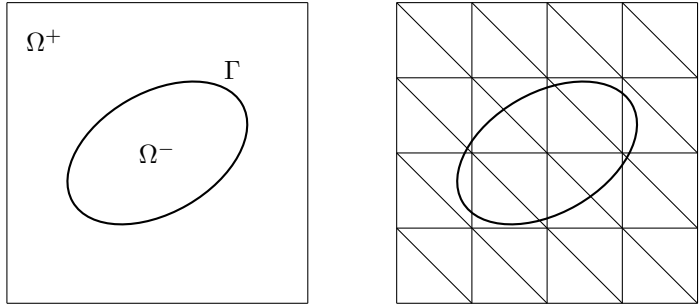

The interface problem arises in many applications. Traditional finite element methods require the mesh to be aligned with the interface to guarantee the optimal convergence rates, see for example \cite{babuvska1970finite,0xu,bramble1996finite,chen1998finite,2014A}. For  complicated interfaces or moving interfaces, unfitted meshes, which are not necessarily aligned with interfaces (see Figure~\ref{interfacepb} for an illustration), have  some advantages over  interface-fitted meshes and have become highly attractive for solving interface problems. The first attempt to use unfitted meshes for solving problems with irregular boundaries dates back to Peskin's immersed boundary method \cite{Peskin1977Numerical} which is in the  finite difference framework.  
%In the finite element framework, traditional finite element methods on unfitted meshes only achieve sub-optimal convergence rates (i.e., $O(h^{1/2})$ in the $H^1$ norm and $O(h)$ in the $L^2$ norm) no matter how high degree of the polynomial is used, see  \cite{}.  
The design and analysis of finite element methods on unfitted meshes with optimal convergence rates was started in \cite{Barrett1984,Barrett1987}. Since then, many unfitted mesh finite element methods have been developed, for example, the unfitted Nitsche's method \cite{hansbo2002unfitted,2012An,2016High,2018Robust}, the extended finite element method \cite{fries2010extended}, the enriched finite element method \cite{Wang2018A},  the multiscale finite method \cite{multi_CHU2010}, the finite element method for high-contrast problems \cite{GuzmanJSC2017} and the immersed finite element (IFE) methods \cite{li1998immersed,Li2003new,li2004immersed,he2012convergence,He2008,Guojcp2020,GuoIMA2019,hou2005numerical}, to name only a few. In this paper, we focus on the IFE method.   The basic idea of IFEs is to modify traditional finite element space to recover the optimal approximation capabilities on unfitted meshes. Differing from other unfitted mesh finite element methods, the IFE method has the same degrees of freedom as that of traditional finite element methods and it can reduce to the traditional finite element method when the interface disappears.  In other words, the IFE spaces are isomorphic to the standard finite element spaces defined on the same mesh, which is an important property for moving interface problems \cite{guo2021SIAM}. 

All the above mentioned unfitted methods are designed for solving the primary variable $u$. In many engineering applications, in contrast to the primary variable $u$,  the flux $\widetilde{\beta}(\boldsymbol{x})\nabla u$ is an important quantity of particular interest. The mixed finite element, see Brezzi and Fortin \cite{Brezzi}, is an efficient method that discretizes the flux variable directly and can preserve fluxes within each element.

However, to the best of our knowledge, there are no known works on IFE methods based on mixed finite elements. This paper is devoted to  develop a mixed finite element method on unfitted meshes following the idea of IFEs. The major challenge is the construction of IFE space for $H(\mathrm{div};\Omega)$ with optimal approximation capabilities and the corresponding theoretical analysis. We consider the well-known mixed method of Raviart-Thomas \cite{RT1977}. Similar to the $P_1$ conforming finite element, functions in the standard Raviart-Thomas finite element space also cannot approximate the exact solution optimally on these elements cut by the interface (called interface elements) due to the interface conditions. By deriving a new interface condition for the flux, we modify the standard Raviart-Thomas shape functions on interface elements to recover the optimal approximation capabilities. We give an explicit formula for the modified shape functions (called IFE shape functions) associated with the degrees of the freedom on edges and prove the unisolvence under a maximum angle condition. Furthermore, the optimal approximation capabilities of the IFE space are derived for the problem with variable coefficients. We also show that the commuting digram for standard Raviart-Thomas elements  also holds for the IFEs. 

One drawback of the IFE space is that it does not belong to $H(\mathrm{div};\Omega)$ because the normal components of IFE functions may be discontinuous across the edges cut by the interface (called interface edges).  One approach to overcome the nonconformity is to add consistent terms locally on interface edges to the bilinear form, and therefore a penalty term should also be included simultaneously to ensure the stability (see \cite{taolin2015siam,2021ji_IFE}). However, for the lowest-order Raviart-Thomas mixed finite element method, we find that the consistent term is zero. Thus, the IFE method is stable without requiring any penalty terms. Unfortunately, we show that the IFE method without penalties or with a conventional penalty term only has suboptimal convergence rates due to the discontinuities of normal components of IFE functions. To overcome the difficulty, we apply an over-penalization only on interface edges which is similar to the approach in \cite{2008Brenner}.  Optimal error estimates are derived rigorously under a slight stronger regularity assumption of the exact solution. The hidden constant in the analysis is independent of the interface location relative to the mesh, which is important for analyzing unfitted methods for interface problems  since interfaces may cut meshes in an arbitrary way.  Some numerical examples are also provided to validated our theoretical analysis.

The paper is organized as follows. In Section~\ref{sec_mix_form}, we introduce the mixed variational formulation and derive interface conditions for the flux. In Section~\ref{sec_IFE}, we first construct the IFE space based on the Raviart-Thomas element and then present the IFE method. In Section~\ref{sec_pro_IFE}, some important properties of the IFE are discussed including the unisolvence of IFE basis functions, the optimal approximation capabilities of the IFE space and the corresponding commuting digram. In Section~\ref{sec_error}, we derive optimal error estimates for the proposed IFE method.  Numerical examples are presented  in Section~\ref{sec_num} to validate our theoretical analysis.   We conclude in the last section.

\section{The mixed variational formulation}\label{sec_mix_form}
By introducing another unknown 
\begin{equation}\label{defp}
\boldsymbol{p}(\boldsymbol{x}):=\widetilde{\beta}(\boldsymbol{x})\nabla u(\boldsymbol{x})\qquad\forall x\in\Omega\backslash\Gamma
\end{equation}
and defining 
\begin{equation*}
\beta(\boldsymbol{x}):=(\widetilde{\beta}(\boldsymbol{x}))^{-1},~ \beta^s(\boldsymbol{x}):=(\widetilde{\beta}^{s}(\boldsymbol{x}))^{-1}, s=+,-,
\end{equation*}
 the interface problem (\ref{p1.1})-(\ref{p1.4})  can be rewritten as 
\begin{align}
\beta\boldsymbol{p}-\nabla u&=0\qquad\mbox{ in } \Omega\backslash\Gamma, \label{p2.1}\\
-\nabla\cdot \boldsymbol{p} &=f  \qquad\mbox{ in } \Omega\backslash\Gamma,\label{p2.2}\\
[u]_{\Gamma}&=0 \qquad\mbox{ on } \Gamma,\label{p2.3}\\
[\boldsymbol{p}\cdot \mathbf{n}]_{\Gamma}&=0 \qquad\mbox{ on } \Gamma,\label{p2.4}\\
u&=0 \qquad\mbox{ on } \partial\Omega.\label{p2.5}
\end{align}
From now on, we use bold letters to denote vector-valued functions. Multiplying the equation (\ref{p2.1}) by a function $\boldsymbol{q}\in H(\mathrm{div};\Omega)$ and integrating by parts, we have
\begin{equation*}
\int_\Omega\beta\boldsymbol{p}\cdot\boldsymbol{q}d\boldsymbol{x}+\int_{\Omega} u\nabla \cdot \boldsymbol{q}d\boldsymbol{x}+\int_\Gamma[u]_\Gamma \boldsymbol{q}\cdot \mathbf{n}ds-\int_{\partial\Omega}u\boldsymbol{q}\cdot\textbf{n}_{\partial\Omega}ds=0,
\end{equation*}
where $\textbf{n}_{\partial\Omega}$ is the unit normal vector of $\partial\Omega$ exterior to $\Omega$, and 
\begin{equation*}
H(\mathrm{div}; \Omega):=\{ \boldsymbol{q}\in (L^2(\Omega))^2 ~: ~\mathrm{div} ~ \boldsymbol{q}\in L^2(\Omega) \}
\end{equation*}
with norm defined by 
\begin{equation*}
\| \boldsymbol{q}\|_{H(\mathrm{div}; \Omega)}^2:=\| \boldsymbol{q} \|^2_{L^2(\Omega)}+\|\mathrm{div}~ \boldsymbol{q}\|^2_{L^2(\Omega)}.
\end{equation*}
It follows from the interface condition (\ref{p2.3}) and the boundary condition (\ref{p2.5})  that
\begin{equation*}
\int_\Omega\beta\boldsymbol{p}\cdot\boldsymbol{q}d\boldsymbol{x}+\int_{\Omega} u\nabla \cdot \boldsymbol{q}d\boldsymbol{x}=0\qquad\forall \boldsymbol{q}\in H(\mathrm{div}; \Omega).
\end{equation*}
Let $\boldsymbol{p}^s=\boldsymbol{p}|_{\Omega^s}$, $s=+,-$. Obviously, $\boldsymbol{p}^s\in H(\mathrm{div};\Omega^s)$, $s=+,-$. By the condition (\ref{p2.4}), we conclude $\boldsymbol{p}\in H(\mathrm{div};\Omega)$. 
Let 
\begin{equation*}
a(\boldsymbol{p},\boldsymbol{q}):=\int_\Omega\beta\boldsymbol{p}\cdot\boldsymbol{q}d\boldsymbol{x},\quad b(\boldsymbol{q},u):=\int_{\Omega} u\nabla \cdot \boldsymbol{q}d\boldsymbol{x}\quad\mbox{ and }\quad F(v):=-\int_{\Omega}fvd\boldsymbol{x},
\end{equation*}
then the mixed variational formulation for the interface problem (\ref{p2.1})-(\ref{p2.5}) reads: 
find $(\boldsymbol{p},u)\in H(\mathrm{div}; \Omega)\times L^2(\Omega)$ such that
\begin{equation}\label{mix_weak}
\begin{aligned}
a(\boldsymbol{p},\boldsymbol{q})+b(\boldsymbol{q},u)&=0 \qquad &&\forall \boldsymbol{q}\in H(\mathrm{div};\Omega),\\
b(\boldsymbol{p},v)&=F(v)\qquad &&\forall v\in L^2(\Omega).
\end{aligned}
\end{equation}
From \cite{Brezzi}, we have the following well-posedness result for this variational problem.
For any $f\in L^2(\Omega)$, the problem (\ref{mix_weak}) has a unique solution $(\boldsymbol{p},u)\in \left(H(\mathrm{div}; \Omega), L^2(\Omega)\right)$ satisfying 
\begin{equation*}
\|\boldsymbol{p}\|_{H(\mathrm{div}; \Omega)}+\|u\|_{L^2(\Omega)}\leq C\|f\|_{L^2(\Omega)}.
\end{equation*}

Next, we investigate the regularity of the solution $(\boldsymbol{p},u)$. Define  
\begin{equation*}
H^m(\Omega^-\cup\Omega^+):=\{ v : v|_{\Omega^s}\in H^m(\Omega^s), s=+,- \}
\end{equation*}
equipped with the norm $\|\cdot\|^2_{H^m(\Omega^+\cup\Omega^-)}:=\|\cdot\|^2_{H^m(\Omega^+)}+\|\cdot\|^2_{H^m(\Omega^-)}$ and the semi-norm $|\cdot|^2_{H^m(\Omega^+\cup\Omega^-)}:=|\cdot|^2_{H^m(\Omega^+)}+|\cdot |^2_{H^m(\Omega^-)}$. Define a subspace of $H^2(\Omega^-\cup\Omega^+)$ as
\begin{equation}\label{def_H2}
\widetilde{H}^2(\Omega):=\{v\in H^2(\Omega^+\cup\Omega^-)  :  [v]_{\Gamma}=0,~~ [\widetilde{\beta}\nabla v\cdot \textbf{n}]_{\Gamma}=0\}.
\end{equation}
It is well-known that the interface problem (\ref{p1.1})-(\ref{p1.4}) has a unique solution 
\begin{equation}\label{regular}
u\in H_0^1(\Omega)\cap \widetilde{H}^2(\Omega) \mbox{ satisfying } \|u\|_{H^2(\Omega^+\cup\Omega^-)}\leq C\|f\|_{L^2(\Omega)},
\end{equation}
where the constant $C$ only depends on $\Omega$, $\Gamma$ and $\widetilde{\beta}$ (see \cite{2012Uniform} for problems with piecewise smooth  coefficients and \cite{huang2002some,multi_CHU2010} for problems with piecewise constant coefficients).

Note that $\widetilde{\beta}^s(\boldsymbol{x})\in C^1(\overline{\Omega^s})$, $s=+,-$. From (\ref{regular}), we immediately have 
\begin{equation}\label{p_h1}
\boldsymbol{p}|_{\Omega^s}\in (H^1(\Omega^s))^2,~ s=+,-.
\end{equation} 
Let $\mathbf{t}(\boldsymbol{x})$ be the unit tangent vector of $\Gamma$ obtained by a $90^\circ$ clockwise rotation of $\mathbf{n}(\boldsymbol{x})$, i.e.,
\begin{equation}\label{rotation}
\mathbf{t}(\boldsymbol{x})=R_{-\frac{\pi}{2}}\mathbf{n}(\boldsymbol{x}), \mbox{ where }
R_\alpha=\begin{bmatrix}
   \cos \alpha   &   -\sin \alpha \\
    \sin \alpha  &    \cos \alpha
\end{bmatrix}.
\end{equation}
From the interface condition (\ref{p2.3}), we know $[\nabla u\cdot\mathbf{t}]_\Gamma=0$ on $\Gamma$, 
which together with (\ref{defp}) yields another interface condition for $\boldsymbol{p}$, i.e.,
\begin{equation}\label{p_tangentjp}
[\beta\boldsymbol{p}\cdot \mathbf{t}]_{\Gamma}(\boldsymbol{x})=0 \qquad\forall x\in \Gamma.
\end{equation}
If we define a subspace of $H(\mathrm{div};  \Omega)$ as
\begin{equation}\label{def_space_div_tilde}
\widetilde{H}^1(\mathrm{div};  \Omega):=\{\boldsymbol{q} \in H(\mathrm{div};\Omega)~:~  \boldsymbol{q}|_{\Omega^s}\in (H^1(\Omega^s))^2,~ s=+,-,~  [\beta\boldsymbol{q}\cdot \mathbf{t}]_{\Gamma}=0\},
\end{equation}
  then it follows  from (\ref{p_h1}) and (\ref{p_tangentjp}) that 
\begin{equation}\label{regular_p}
\boldsymbol{p}\in \widetilde{H}^1(\mathrm{div};  \Omega).
\end{equation}
\section{The immersed Raviart-Thomas mixed finite element method}\label{sec_IFE}
Let $ \{\mathcal{T}_h\}_{h>0}$ be a family of conforming triangulations of $\Omega$ with meshsize  $h:=\max_{T\in\mathcal{T}_h}h_T$, where $h_T$ is the diameter of $T\in\mathcal{T}_h$. We assume that $\mathcal{T}_h$ is shape regular, i.e., for every $T$, there exists a positive constant  $\varrho$ such that  $ h_T\leq \varrho r_T$ where $r_T$ is the diameter of the largest circle inscribed in $T$.  Denote $\mathcal{E}_h$ as the set of edges of the triangulation, and let $\mathcal{E}^b_h:=\{e\in\mathcal{E}_h : e\subset\partial\Omega\}$, $\mathcal{E}^\circ_h:=\mathcal{E}_h\backslash\mathcal{E}^b_h$. We adopt the convention that elements $T\in\mathcal{T}_h$ and edges $e\in\mathcal{E}_h$ are open sets. 
Then, the sets of interface elements and interface edges are defined as
\begin{equation*}
\mathcal{T}_h^\Gamma :=\{T\in\mathcal{T}_h :  T\cap \Gamma\not = \emptyset\}, \qquad\mathcal{E}_h^\Gamma:=\{e\in \mathcal{E}_h : e\cap\Gamma\not=\emptyset \}.
\end{equation*}
We also assume that all interface triangles satisfy the maximum angle condition $\alpha_{max}\leq \pi/2$ which is a sufficient condition for the unisolvence of IFE functions (see Lemma~\ref{lem_basis}). We emphasize that the maximum angle assumption does not restrict the application of IFE methods since we can simply use Cartesian meshes regardless of the location of the interface, which is an advantage over the interface-fitted mesh method.
The set of non-interface elements is denoted by  $\mathcal{T}^{non}_h:=\mathcal{T}_h\backslash\mathcal{T}_h^{\Gamma}$. 
We can alway refine the mesh to  satisfy the following assumption.
\begin{assumption}\label{assum_2}
The interface $\Gamma$ does not intersect the boundary of any interface element at more than two points. The interface $\Gamma$ does not intersect  the closure $\overline{e}$ for any $e\in\mathcal{E}_h$  at more than one point.
\end{assumption}

The interface $\Gamma$ is approximated by $\Gamma_h$ that is composed of all the line segments connecting the intersection points of the triangulation and the interface. In addition, we assume that the approximated interface $\Gamma_h$ divides $\Omega$ into two disjoint sub-domains $\Omega^+_h$ and $\Omega^-_h$ such that $\Gamma_h=\partial \Omega^-_h$.   Let $\textbf{n}_h(\boldsymbol{x})$  be the unit normal vector of $\Gamma_h$ pointing toward $\Omega^+_h$.  The unit  tangent vector of $\Gamma_h$ can be obtained by a $90^\circ$ clockwise rotation of $\mathbf{n}_h(\boldsymbol{x})$, i.e.,
\begin{equation}\label{disth}
\textbf{t}_h(\boldsymbol{x})=R_{-\frac{\pi}{2}}\textbf{n}_h(\boldsymbol{x}).
\end{equation}

\subsection{The IFE space}
On each $T\in\mathcal{T}_h$, define the traditional local Raviart-Thomas space with the lowest-order
\begin{equation*}
\mathcal{RT}(T)=\left\{\boldsymbol{\phi}  : \boldsymbol{\phi}(x_1,x_2)= \begin{bmatrix}
a_T\\
c_T
\end{bmatrix}
+b_T
\begin{bmatrix}
x_1\\
x_2
\end{bmatrix},~ a_T, b_T, c_T \in \mathbb{R}^1
 \right\}.
\end{equation*}
To get the optimal approximation capabilities on interface elements $T\in\mathcal{T}_h^\Gamma$, we need to modify the local Raviart-Thomas space  $\mathcal{RT}(T)$ according to  the  interface conditions of the exact solution $\boldsymbol{p}$. 

Let $T_h^+:=T\cap\Omega_h^+$ and  $T_h^-:=T\cap\Omega_h^-$ for all $T\in\mathcal{T}_h^\Gamma$, see Figure~\ref{interface_ele} for an illustration.
We define a local immersed Raviart-Thomas space  $\mathcal{IRT}(T)$ as the set of the following functions 
\begin{equation}\label{IFE_shape}
\boldsymbol{\phi}(\boldsymbol{x})=\left\{
\begin{aligned}
&\boldsymbol{\phi}^+(\boldsymbol{x})\in \mathcal{RT}(T)~~ \mbox{ if } \boldsymbol{x}\in T_h^+,\\
&\boldsymbol{\phi}^-(\boldsymbol{x})\in \mathcal{RT}(T) ~~ \mbox{ if } \boldsymbol{x}\in T_h^-,
\end{aligned}
\right.
\end{equation}
satisfying 
\begin{align}
&[\boldsymbol{\phi}\cdot \textbf{n}_h]_{\Gamma_h\cap T}(\boldsymbol{x}):=(\boldsymbol{\phi}^+\cdot \textbf{n}_h)(\boldsymbol{x})-(\boldsymbol{\phi}^-\cdot \textbf{n}_h)(\boldsymbol{x})=0\qquad \forall \boldsymbol{x}\in \Gamma_h\cap T,\label{IFE_con1}\\
&[\beta_T\boldsymbol{\phi}\cdot \textbf{t}_h]_{\Gamma_h\cap T}(\boldsymbol{x}_T):=(\beta_T^+ \boldsymbol{\phi}^+\cdot \textbf{t}_h)(\boldsymbol{x}_T)-(\beta_T^- \boldsymbol{\phi}^-\cdot \textbf{t}_h)(\boldsymbol{x}_T)=0,\label{IFE_con2}\\
&\nabla \cdot \boldsymbol{\phi}^+-\nabla \cdot \boldsymbol{\phi}^-=0.\label{IFE_con3}
\end{align}
Here $\boldsymbol{x}_T$ is an arbitrary point on $\Gamma_h\cap T$,  and $\beta_T(\boldsymbol{x})$ is a piecewise constant which is defined by $\beta_T(\boldsymbol{x})|_{T_h^s}=\beta_T^s$, $s=+,-$. The constants $\beta_T^+$ and $\beta_T^+$ are chosen such that 
\begin{equation}\label{cond_beta}
\|\beta^s(\boldsymbol{x})-\beta_T^s\|_{L^\infty(T\cap\Omega^s)}\leq Ch,~~~s=+,-.
\end{equation}
Actually, we can choose $\beta_T^s=\beta^s(\boldsymbol{x}_T^s)$ with arbitrary points $\boldsymbol{x}_T^s\in T\cap \Omega^s$, $s=+,-$ to satisfy the requirement (\ref{cond_beta}) since $\beta^s(\boldsymbol{x})\in C^1(\overline{\Omega^s})$, $s=+,-$.
\begin{remark}
 Note that $\boldsymbol{\phi}^s$, $s=+,-$ have the property that their normal components along any straight lines are constants. Thus, (\ref{IFE_con1}) only provides one condition although the equality is enforced  over the entire line segment $\Gamma_h\cap T$.  The first two conditions (\ref{IFE_con1}) and (\ref{IFE_con2}) are inspired by (\ref{p2.4}) and (\ref{p_tangentjp}) for continuous problems. 
However, the third condition (\ref{IFE_con3}) is added only for the unisolvence of the IFE basis functions and does not provide any approximation capabilities since we use the lowest-order Raviart-Thomas elements.
\end{remark}

%If we write $\boldsymbol{\phi}^s$, $s=+,-$ as
%\begin{equation}
%\boldsymbol{\phi}^s= \begin{bmatrix}
%a_T^s\\
%c_T^s
%\end{bmatrix}
%+b_T^s
%\begin{bmatrix}
%x_1\\
%x_2
%\end{bmatrix},
%\end{equation}
%then we can see that there are six unknowns $a_T^+,  b_T^+,  c_T^+, a_T^-,  b_T^-,  c_T^-$. 

On each element $T\in\mathcal{T}_h$, the local degrees of freedom are defined as
\begin{equation}\label{dof}
N_{i,T}(\boldsymbol{\phi}):=\frac{1}{|e_i|}\int_{e_i}\boldsymbol{\phi}\cdot\textbf{n}_{i,T}ds,\qquad i=1,2,3,
\end{equation}
where $e_i$, $i=1,2,3$ are edges of the element $T$,  $|e_i|$ denotes the length of the edge $e_i$, and $\textbf{n}_{i,T}$ is the unit normal vector of $e_i$ exterior to $T$. The global IFE space $\mathcal{IRT}(\mathcal{T}_h)$  is then defined as the set of all functions satisfying 
\begin{equation*}\left\{
\begin{aligned}
&\boldsymbol{\phi}|_{T}\in \mathcal{RT}(T) \quad&&\forall T\in\mathcal{T}_h^{non},\\
&\boldsymbol{\phi}|_{T}\in \mathcal{IRT}(T)\quad &&\forall T\in\mathcal{T}_h^\Gamma,\\
&\int_{e}[\boldsymbol{\phi}\cdot\textbf{n}_e]_eds=0\qquad &&\forall e\in\mathcal{E}_h^\circ.
\end{aligned}
\right.
\end{equation*}
Here the jump across an edge $e$ is defined by
\begin{equation*}
[\boldsymbol{\phi}\cdot\textbf{n}_e]_e:=(\boldsymbol{\phi}|_{T_1^e}-\boldsymbol{\phi}|_{T_2^e})\cdot\textbf{n}_e,
\end{equation*}
where $\overline{T_1^e}\cap \overline{T_2^e}= \overline{e} $  and $\textbf{n}_e$ is the unit normal vector of the edge $e$ exterior to $T_1^e$.

\begin{remark} 
%For any $e\in\mathcal{E}_h^\Gamma$, let  $T_i^e\in\mathcal{T}_h^\Gamma$, $i=1,2$ be the interface elements with $\overline{e}=\overline{T_1^e}\cap\overline{T_2^e}$. 
For any $\boldsymbol{\phi}\in \mathcal{IRT}(\mathcal{T}_h)$ and $e\in\mathcal{E}_h^\Gamma$,  it is easy to see that $\boldsymbol{\phi}|_{T_1^e}\cdot\textbf{n}_{e}$ and $\boldsymbol{\phi}|_{T_2^e}\cdot\textbf{n}_{e}$ are piecewise constants on the edge $e$.  Thus, the condition $\int_{e}[\boldsymbol{\phi}\cdot\textbf{n}_e]_eds=0$ cannot imply
$[\boldsymbol{\phi}\cdot\textbf{n}_{e}]_e=0$.  In other words, $\boldsymbol{\phi}\cdot\textbf{n}_{e}$ may be discontinuous on  all interface edges $e\in\mathcal{E}_h^\Gamma$. Hence, we conclude that the IFE space is nonconforming, i.e.,
$\mathcal{IRT}(\mathcal{T}_h)\not\subset H(\mathrm{div};  \Omega).$
\end{remark}

\subsection{The IFE method}
We define 
\begin{equation}\label{betamaxmin}
\beta_{min}:=\widetilde{\beta}^{-1}_{max},\quad\beta_{max}:=\widetilde{\beta}^{-1}_{min},
\end{equation}
and extend the coefficients $\beta^s(\boldsymbol{x})$, $s=+,-$ smoothly to slight larger domains 
$$\Omega_e^s:=\{x\in \overline{T} :  ~\forall T\in\mathcal{T}_h \mbox{ and } T\cap \Omega^s\not=\emptyset\}, ~s=+,-,$$ such that 
\begin{equation}\label{ext_beta}
\beta^s(\boldsymbol{x})\in C^1(\overline{\Omega_e^s})~~\mbox{ and }~~ \beta_{min} \leq \beta^s(\boldsymbol{x})\leq \beta_{max},~~s=+,-. 
\end{equation}
Thus, there exists a constant $C_\beta$ such that
\begin{equation}\label{new_vari_deri}
\|\nabla \beta^s\|_{L^\infty(\overline{\Omega_e^s})}\leq C_\beta,\qquad s=+,-.
\end{equation}
Note that, if $\beta$ is a piecewise constant function, it holds $C_\beta=0$.
For simplicity of the implementation, we approximate the coefficient $\beta(\boldsymbol{x})$ by
\begin{equation*}
\beta_h(\boldsymbol{x})=\left\{
\begin{aligned}
&\beta^+(\boldsymbol{x}) \quad\mbox{ if } \boldsymbol{x}\in\Omega_h^+,\\
&\beta^-(\boldsymbol{x}) \quad\mbox{ if } \boldsymbol{x}\in\Omega_h^-.\\ 
\end{aligned}\right.
\end{equation*}
Define a piecewise constant space 
\begin{equation*}
M_h:= \{ v\in L^2(\Omega) : v|_T\in \mathbb{R}^1~~\forall T\in\mathcal{T}_h\},
\end{equation*}
and the following discrete bilinear forms
\begin{equation*}
\begin{aligned}
&A_h(\boldsymbol{p_h},\boldsymbol{q}_h):=a_h(\boldsymbol{p_h},\boldsymbol{q}_h)+s_h(\boldsymbol{p_h},\boldsymbol{q}_h), \quad b_h(\boldsymbol{q}_h,u_h):=\int_{\Omega} u_h\nabla_h \cdot \boldsymbol{q}_hd\boldsymbol{x}, \\
&a_h(\boldsymbol{p_h},\boldsymbol{q}_h):=\int_\Omega\beta_h\boldsymbol{p}_h\cdot\boldsymbol{q}_hd\boldsymbol{x}, \quad s_h(\boldsymbol{p_h},\boldsymbol{q}_h):=\eta \sum_{e\in\mathcal{E}_h^\Gamma}\int_e [\boldsymbol{p}_h\cdot\textbf{n}_e]_e[\boldsymbol{q}_h\cdot\textbf{n}_e]_eds,
\end{aligned}
\end{equation*}
where $\eta>0$ is a penalty parameter independent of $h$  and  $\nabla_h\cdot$ is understood in a piecewise sense, i.e.,  $(\nabla_h \cdot \boldsymbol{q}_h)|_T=\nabla  \cdot \boldsymbol{q}_h|_{T}$ for all $T\in\mathcal{T}_h$ since $\boldsymbol{q}_h\not\in H(\mathrm{div}; \Omega)$.  The  immersed Raviart-Thomas mixed finite element method reads: find $(\boldsymbol{p}_h,u_h)\in \mathcal{IRT}(\mathcal{T}_h)\times M_h$ such that
\begin{equation}\label{mix_IFE}
\begin{aligned}
A_h(\boldsymbol{p}_h,\boldsymbol{q}_h)+b_h(\boldsymbol{q}_h,u_h)&=0 \qquad &&\forall \boldsymbol{q}_h\in \mathcal{IRT}(\mathcal{T}_h),\\
b_h(\boldsymbol{p}_h,v_h)&=F(v_h)\qquad &&\forall v_h\in M_h.
\end{aligned}
\end{equation}

The above IFE method  is inconsistent. There are two kinds of inconsistent errors, which are shown in the following lemma. 
 \begin{lemma}
 Let $(\boldsymbol{p},u)$ and $(\boldsymbol{p}_h,u_h)$ be the solutions of problems (\ref{mix_weak}) and (\ref{mix_IFE}), respectively. Then it holds that, for all $\boldsymbol{q}_h\in \mathcal{IRT}(\mathcal{T}_h)$,
 \begin{equation}\label{inconsis}
 \begin{aligned}
 A_h(\boldsymbol{p}-\boldsymbol{p}_h,&\boldsymbol{q}_h)+b_h(\boldsymbol{q}_h,u-u_h)=\left(a_h(\boldsymbol{p},\boldsymbol{q}_h)-a(\boldsymbol{p},\boldsymbol{q}_h)\right)+\sum_{e\in\mathcal{E}_h^\Gamma}\int_eu[\boldsymbol{q}_h\cdot\textbf{n}_{e}]_eds.
 \end{aligned}
 \end{equation}
 \end{lemma}
\begin{proof}
Multiplying the equation (\ref{p2.1}) by a function $\boldsymbol{q}_h\in \mathcal{IRT}(\mathcal{T}_h)$ and  integrating by parts on each elements yield 
\begin{equation}\label{pro_inconsis}
\int_\Omega\beta\boldsymbol{p}\cdot\boldsymbol{q}_hd\boldsymbol{x}+\sum_{T\in\mathcal{T}_h}\int_{T} u\nabla \cdot \boldsymbol{q}_hd\boldsymbol{x} -\sum_{e\in\mathcal{E}_h^\Gamma}\int_eu[\boldsymbol{q}_h\cdot\textbf{n}_e]_eds  =0,
\end{equation}
where we have used  (\ref{p2.3}), (\ref{p2.5}) and the fact that the normal component of $\boldsymbol{q}_h$ is only discontinuous on interface edges $e\in \mathcal{E}_h^\Gamma$. Note that $s_h(\boldsymbol{p}-\boldsymbol{p}_h, \boldsymbol{q}_h)=0$, the desired result (\ref{inconsis}) is obtained by subtracting the first equation of (\ref{mix_IFE}) from (\ref{pro_inconsis}).
\end{proof}

\begin{remark}
Different from the interior penalty method, the consistent term cannot be added into the bilinear form because 
$$\sum_{e\in\mathcal{E}_h^\Gamma}\int_eu_h[\boldsymbol{q}_h\cdot\textbf{n}_e]_eds =0 \quad\forall u_h\in M_h ~\forall \boldsymbol{q}_h\in  \mathcal{IRT}(\mathcal{T}_h).$$
We find that if we use a conventional penalty $$\eta \sum_{e\in\mathcal{E}_h^\Gamma}|e|\int_e [\boldsymbol{p}_h\cdot\textbf{n}_e]_e[\boldsymbol{q}_h\cdot\textbf{n}_e]_eds,$$ then the IFE method cannot achieve the optimal convergence rates. 
Guided by \cite{2008Brenner},  we apply the over-penalization $s_h(\cdot,\cdot)$ to overcome the difficulty. The IFE method (\ref{mix_IFE})  is stable for any choice of the penalty parameter even if the penalty parameter $\eta=0$. However,  we need to choose a positive $\eta$ independent of $h$ to ensure the optimal convergence rates (see Theorem~\ref{theo_error}).
\end{remark}

\section{Properties of the IFE space}\label{sec_pro_IFE}
In this section, we discuss  some important properties of the IFE space $\mathcal{IRT}(\mathcal{T}_h)$. We begin with some interpolation operators.
On each element $T\in\mathcal{T}_h$, define a local interpolation operator  $\Pi_{T}: W(T)\rightarrow \mathcal{RT}(T)$ such that
\begin{equation*}
N_{i,T}(\Pi_{T} \boldsymbol{q})=N_{i,T}(\boldsymbol{q}),\quad i=1,2,3,
\end{equation*}
where $W(T)= H(\mathrm{div}; T)\cap (L^s(T))^2$ with a fixed $s>2$.
Similarly, on each interface element $T\in\mathcal{T}_h^\Gamma$, define  $I^{IFE}_{h,T}: W(T)\rightarrow \mathcal{IRT}(T)$ such that
\begin{equation}\label{local_IFE_inter}
N_{i,T}(\Pi_{T}^{IFE} \boldsymbol{q})=N_{i,T}(\boldsymbol{q}),\quad i=1,2,3.
\end{equation}
 The global IFE interpolation operator now is $\Pi_h^{IFE}: W(\Omega) \rightarrow \mathcal{IRT}(\mathcal{T}_h)$  such that 
\begin{equation}\label{inter_IFEoper_def}
 (\Pi_h^{IFE}v)|_{T}=\left\{
\begin{aligned}
&\Pi_{T}^{IFE}v\quad&&\mbox{ if } ~T\in\mathcal{T}_h^\Gamma,\\
&\Pi_{T} v&&\mbox{ if }~ T\in\mathcal{T}_h^{non},\\
\end{aligned}\right.
\end{equation}
where $W(\Omega)= H(\mathrm{div}; \Omega)\cap (L^s(\Omega))^2$ with a fixed $s>2$.
 We also define the standard Raviart-Thomas finite element space 
\begin{equation}\label{rt_space}
\mathcal{RT}(\mathcal{T}_h)=\{\boldsymbol{q}_h\in H(\mathrm{div};  \Omega) :  \boldsymbol{q}_h\in \mathcal{RT}(T) \}
\end{equation}
and a corresponding interpolation operator $\Pi_h: W(\Omega) \rightarrow \mathcal{RT}(\mathcal{T}_h)$ such that 
\begin{equation}\label{inter_oper_def}
(\Pi_hv)|_{T}=\Pi_{T} v \quad\forall T\in\mathcal{T}_h.
\end{equation}

Note that  the local interpolation operator $\Pi_T$ is well-defined because $\boldsymbol{\phi}\in \mathcal{RT}(T)$ is uniquely determined by $N_{i,T}(\boldsymbol{\phi})$, $i=1,2,3$.  And we can define the standard Raviart-Thomas basis functions on each interface element $T\in\mathcal{T}_h^\Gamma$ as
\begin{equation}\label{basis_standard}
\boldsymbol{\lambda}_{i,T}\in \mathcal{RT}(T),~~\quad  N_{j,T}(\boldsymbol{\lambda}_{i,T})=\delta_{ij},~~ i,j=1,2,3,
\end{equation}
where $\delta_{ij}$ is the Kronecker function.

However, the well-definedness for the interpolation operator $\Pi^{IFE}_T$ is not obvious. We need the unisolvence of IFE shape functions in $\mathcal{IRT}(T)$, which is proved in the following subsection.

\subsection{The unisolvence of IFE shape functions}
Without loss of generality, we consider an interface element $T\in\mathcal{T}_h^\Gamma$. Given a function $\boldsymbol{\phi}\in\mathcal{IRT}(T)$, we define a function $\boldsymbol{\phi}^0$ such that
\begin{equation}\label{phi0}
 \boldsymbol{\phi}^0\in \mathcal{RT}(T),~~N_{i,T}(\boldsymbol{\phi}^0)=N_{i,T}(\boldsymbol{\phi}),~ i=1,2,3.
 \end{equation}
Obviously,
\begin{equation}\label{phi0_con}
[\boldsymbol{\phi}^0\cdot\textbf{n}_h]_{\Gamma_h\cap T}=0,~[\boldsymbol{\phi}^0\cdot\textbf{t}_h]_{\Gamma_h\cap T}=0, ~\nabla\cdot \boldsymbol{\phi}^0|_{T_h^+}-\nabla\cdot \boldsymbol{\phi}^0|_{T_h^-}=0.
\end{equation}
The function $\boldsymbol{\phi}^0$ is unisolvent and can be expressed by the standard Raviart-Thomas basis functions
\begin{equation}\label{expressphi0}
\boldsymbol{\phi}^0=\sum_{i=1}^3N_{i,T}(\boldsymbol{\phi})\boldsymbol{\lambda}_{i,T}
\end{equation}
We define another function $\boldsymbol{\phi}^{J_t}$ such that
\begin{equation}\label{phijt}
\begin{aligned}
&\boldsymbol{\phi}^{J_t}|_{T_h^+}\in \mathcal{RT}(T),  ~~\boldsymbol{\phi}^{J_t}|_{T_h^-}\in \mathcal{RT}(T),~N_{i,T}(\boldsymbol{\phi}^{J_t})=0,~ i=1,2,3,\\
&[\boldsymbol{\phi}^{J_t}\cdot \textbf{n}_h]_{\Gamma_h\cap T}=0,~[\boldsymbol{\phi}^{J_t}\cdot \textbf{t}_h]_{\Gamma_h\cap T}(\boldsymbol{x}_T)=1,~\nabla \cdot \boldsymbol{\phi}^{J_t}|_{T_h^+}-\nabla \cdot \boldsymbol{\phi}^{J_t}|_{T_h^-}=0,
\end{aligned}
\end{equation}
where the point $\boldsymbol{x}_T\in \Gamma_h\cap T$ is the same as that in (\ref{IFE_con2}). The function $\boldsymbol{\phi}^{J_t}$ defined above is also unisolvent which is proved below.
Suppose there is another function satisfying (\ref{phijt}), denoted by $\boldsymbol{\phi}_1^{J_t}$. From (\ref{phijt}), it is easy to show that $\boldsymbol{\phi}^{J_t}-\boldsymbol{\phi}_1^{J_t}=0$ which implies the uniqueness.  The existence can be proved by constructing the function as follows,
\begin{equation}\label{phijt_c}
\boldsymbol{\phi}^{J_t}=\boldsymbol{\omega}-\Pi_{h,T}\boldsymbol{\omega},~~~\boldsymbol{\omega}=\left\{
\begin{aligned}
&\boldsymbol{\omega}^+=\textbf{t}_h\qquad&&\mbox{ in } T_h^+,\\
&\boldsymbol{\omega}^-=\mathbf{0}&&\mbox{ in } T_h^-.
\end{aligned}\right.
\end{equation}
It is easy to verify that the function in (\ref{phijt_c}) indeed satisfies (\ref{phijt}). 
\begin{lemma}
Given $\boldsymbol{\phi}\in \mathcal{IRT}(T)$, if we know the jump 
\begin{equation}\label{phiji_known}
\mu:=[\boldsymbol{\phi}\cdot\textbf{t}_h]_{\Gamma_h\cap T}(\boldsymbol{x}_T),
\end{equation}
 then  the function $\boldsymbol{\phi}$ can be written as 
\begin{equation}\label{decomp1}
\boldsymbol{\phi}=\boldsymbol{\phi}^0+\mu\boldsymbol{\phi}^{J_t}.
\end{equation}
\end{lemma}
\begin{proof}
 Let $\boldsymbol{w}=\boldsymbol{\phi}^0+\mu\boldsymbol{\phi}^{J_t}-\boldsymbol{\phi}$. We just need to prove that $\boldsymbol{w}=0$. It follows from  (\ref{IFE_con1}), (\ref{IFE_con3}), (\ref{phi0_con}) and (\ref{phijt}) that
\begin{equation*}
[\boldsymbol{w}\cdot\textbf{n}_h]_{\Gamma_h\cap T}=0,~[\boldsymbol{w}\cdot\textbf{t}_h]_{\Gamma_h\cap T}(\boldsymbol{x}_T)=0, ~\nabla\cdot \boldsymbol{w}|_{T_h^+}-\nabla\cdot \boldsymbol{w}|_{T_h^-}=0,
\end{equation*}
which together with $\boldsymbol{w}|_{T_h^s}\in\mathcal{RT}(T)$, $s=+,-$ implies 
\begin{equation}\label{winrt}
\boldsymbol{w}\in \mathcal{RT}(T).
\end{equation}
On the other hand,  we know from (\ref{phi0}) and (\ref{phijt}) that
\begin{equation}\label{nw0}
N_{i,T}(\boldsymbol{w})=N_{i,T}(\boldsymbol{\phi}^0)+\mu N_{i,T}(\boldsymbol{\phi}^{J_t})-N_{i,T}(\boldsymbol{\phi})=0,~i=1,2,3.
\end{equation}
Combining (\ref{winrt}) and (\ref{nw0}), we conclude $\boldsymbol{w}=0$, which completes the proof.
\end{proof}

Now the problem is to find the corresponding jump $\mu$ so that the condition (\ref{IFE_con2}) is satisfied. 
Substituting (\ref{decomp1}) into (\ref{IFE_con2}), we get the following equation for the jump $\mu$,
\begin{equation}\label{eq_lambda0}
[\beta_T\boldsymbol{\phi}^{J_t}\cdot \textbf{t}_h]_{\Gamma_h\cap T}(\boldsymbol{x}_T)\mu=-[\beta_T\boldsymbol{\phi}^0\cdot \textbf{t}_h]_{\Gamma_h\cap T}(\boldsymbol{x}_T).
\end{equation}
By (\ref{phijt_c}) and (\ref{phi0_con}), we find
\begin{equation*}
\begin{aligned}
&[\beta_T\boldsymbol{\phi}^{J_t}\cdot \textbf{t}_h]_{\Gamma_h\cap T}(\boldsymbol{x}_T)=\beta_T^+-(\beta_T^+-\beta_T^-)(\Pi_{h,T}\boldsymbol{\omega})(\boldsymbol{x}_T)\cdot\textbf{t}_h,\\
&-[\beta_T\boldsymbol{\phi}^0\cdot \textbf{t}_h]_{\Gamma_h\cap T}(\boldsymbol{x}_T)=-(\beta_T^+-\beta_T^-)\boldsymbol{\phi}^0(\boldsymbol{x}_T)\cdot \textbf{t}_h.
\end{aligned}
\end{equation*}
Hence, the equation (\ref{eq_lambda0}) can be simplified as
\begin{equation}\label{eq_lambda1}
\left(1+(\beta_T^-/\beta_T^+-1)(\Pi_{h,T}\boldsymbol{\omega})(\boldsymbol{x}_T)\cdot\textbf{t}_h\right)\mu=(\beta_T^-/\beta_T^+-1)\boldsymbol{\phi}^0(\boldsymbol{x}_T)\cdot \textbf{t}_h.
\end{equation}

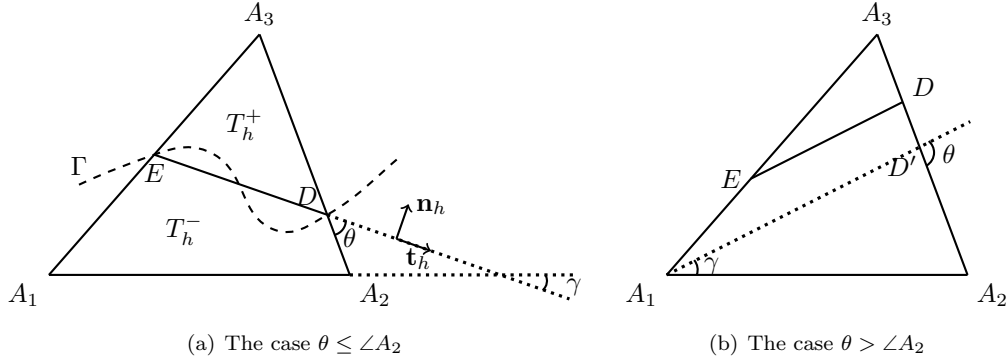
\begin{figure}[htbp]
\centering
\subfigure[The case $\theta\leq \angle A_2$]{ \label{interface_ele1} 
\begin{tikzpicture}[scale=2]
\draw [thick] (-1,-1)--(1,-1);
\draw [thick, name path=e2] (-1,-1)--(0.4,0.6);
\draw [thick, name path=e1] (0.4,0.6)--(1,-1);
\draw [thick](-1+1.4/2,-1+1.6/2)--(1-0.6/4,-1+1.6/4);
\node [below left] at (-1,-1) {$A_1$};
\node [below right ]at (1,-1) {$A_2$};
\node [above]at (0.4,0.6) {$A_3$};
\node [above left] at (1-0.6/4,-1+1.6/4) {$D$};
\node [below ] at (-1+1.4/2,-1+1.6/2) {$E$};
\draw [dashed, thick](-1+1.4/2-0.5,-1+1.6/2-0.2) to [out=25,in=180+20] (-1+1.4/2,-1+1.6/2)
to [out=20,in=150] (-1+1.4/2+0.4,-1+1.6/2)
to [out=-30,in=160] (-1+1.4/2+0.8,-1+1.6/2-0.5)
to [out=-20,in=180+30] (1-0.6/4,-1+1.6/4)
to [out=30,in=180+40] (1-0.6/4+0.5,-1+1.6/4+0.4);
\node [above] at (-1+1.4/2-0.5,-1+1.6/2-0.2) {$\Gamma$};
\node at (-1+1.4/2+0.6,-1+1.6/2+0.2) {$T_h^+$};
\node at (-1+1.4/2+0.2,-1+1.6/2-0.5) {$T_h^-$};
\draw [dotted, very thick] (0.85,-0.6)--(0.85+1.15*1.4,-0.6+-0.4*1.4);
\draw [dotted, very thick] (1,-1)--(2.5,-1);
\draw [ ->,thick] (0.85+1.15*0.4,-0.6-0.4*0.4)--(0.85+1.15*0.6,-0.6-0.4*0.6);
\draw [ ->,thick] (0.85+1.15*0.4,-0.6-0.4*0.4)--(0.85+1.15*0.4+0.4*0.2 ,-0.6-0.4*0.4 +1.15*0.2);
\node [above right] at (0.85+1.15*0.4+0.07,-0.6-0.4*0.4+0.1) {$\textbf{n}_h$};
\node [below right] at (0.85+1.15*0.4,-0.6-0.4*0.4) {$\textbf{t}_h$};
\draw [thick] (1-0.6/6,-1+1.6/6) arc [radius=0.1, start angle=-75, end angle= 0];
\node at (1-0.6/4+0.15,-1+1.6/4-0.15) {$\theta$};
\draw [thick] (2+0.3,-1)arc [radius=0.08, start angle=30, end angle= -50];
\node at (2+0.5-0.01,-1-0.08){$\gamma$};
\end{tikzpicture}
}
\subfigure[The case $\theta> \angle A_2$]{ \label{interface_ele2} 
\begin{tikzpicture}[scale=2]
\draw [thick] (-1,-1)--(1,-1);
\draw [thick, name path=e2] (-1,-1)--(0.4,0.6);
\draw [thick, name path=e1] (0.4,0.6)--(1,-1);
\draw [thick](-1+1.4/2.5,-1+1.6/2.5)--(1-0.6/1.6*1.15,-1+1.6/1.6*1.15);
\node [below left] at (-1,-1) {$A_1$};
\node [below right ]at (1,-1) {$A_2$};
\node [above]at (0.4,0.6) {$A_3$};
\node [ right] at (1-0.6/1.6*1.15,-1+1.6/1.6*1.15+0.1) {$D$};
\node [left ] at (-1+1.4/2.5,-1+1.6/2.5) {$E$};
\draw [very thick,dotted] (-1,-1)--(-1+1.345*1.5,-1+0.68*1.5);
\draw [thick] (-1+0.2,-1) arc [radius=0.1, start angle=-20, end angle= 40];
\node at (-1+0.3,-1+0.06) {$\gamma$};
\draw [thick] (1-0.6/2.2,-1+1.6/2.2) arc [radius=0.1, start angle=-60, end angle= 50];
\node at (1-0.6/2.2+0.15,-1+1.6/2.2+0.08){$\theta$};
\node [left ] at (1-0.6/2.2,-1+1.6/2.2) {$D^\prime$};
\end{tikzpicture}
}
 \caption{Typical interface elements }\label{interface_ele}
\end{figure}

\begin{lemma}\label{lem_01}
Let $\boldsymbol{\omega}$ be defined in (\ref{phijt_c}) and $\alpha_{max}$ be the maximum angle of the triangle $T\in\mathcal{T}_h^\Gamma$. If $\alpha_{max}\leq \pi/2$,  then it holds  
\begin{equation}\label{est01}
0\leq (\Pi_{h,T}\boldsymbol{\omega})(\boldsymbol{x})\cdot\textbf{t}_h\leq 1\qquad \forall \boldsymbol{x}\in T.
\end{equation}
\end{lemma}
\begin{proof}
For clarity, we consider a typical interface element $T=\triangle A_1A_2A_3$ with $e_1=\overline{A_2A_3}$, $e_2=\overline{A_1A_3}$ and $e_3=\overline{A_1A_2}$. Without loss of generality, we  assume that the interface $\Gamma$ cuts $e_1$ and $e_2$ at points $D$ and $E$.  There are two cases: Case 1: $T_h^+=\triangle EDA_3$ (as shown in Figure~\ref{interface_ele}); Case 2: $T_h^-=\triangle EDA_3$. In Case 1, we have 
\begin{equation}\label{case1}
\textbf{t}_h=|DE|^{-1}\overrightarrow{ED}~~\mbox{ and }~~\boldsymbol{\omega}=\left\{
\begin{aligned}
&|DE|^{-1}\overrightarrow{ED}\qquad&&\mbox{ in } \triangle EDA_3,\\
&\mathbf{0}&&\mbox{ in } T\backslash\triangle EDA_3.
\end{aligned}\right.
\end{equation}
In order to distinguish these two cases, we replace the notations $\textbf{t}_h$ and $\boldsymbol{\omega}_h$ by $\textbf{t}^\prime_h$ and $\boldsymbol{\omega}^\prime$ in Case 2. According to (\ref{disth}), we have the following relations according to (\ref{phijt_c}) 
\begin{equation}\label{case2}
\textbf{t}^\prime_h=|DE|^{-1}\overrightarrow{DE}~~\mbox{ and }~~\boldsymbol{\omega}^\prime=\left\{
\begin{aligned}
&\mathbf{0}\qquad&&\mbox{ in } \triangle EDA_3,\\
&|DE|^{-1}\overrightarrow{DE}\qquad&&\mbox{ in } T\backslash\triangle EDA_3.
\end{aligned}\right.
\end{equation}
Comparing  (\ref{case1}) with (\ref{case2}), we find
\begin{equation*}
\textbf{t}^\prime_h=-\textbf{t}_h,~~\boldsymbol{\omega}^\prime=\boldsymbol{\omega}-\textbf{t}_h,
\end{equation*}
which implies 
\begin{equation}\label{case12c}
(\Pi_{h,T}\boldsymbol{\omega}^\prime)\cdot\textbf{t}^\prime_h=(\Pi_{h,T}\boldsymbol{\omega}-\textbf{t}_h)\cdot(-\textbf{t}_h)=1-(\Pi_{h,T}\boldsymbol{\omega})\cdot\textbf{t}_h.
\end{equation}
If the estimate~(\ref{est01}) holds for Case 1, then we can conclude from (\ref{case12c})  that the estimate~(\ref{est01}) also holds for Case 2. Therefore, we only need to consider Case 1 whose geometric configuration is given in Figure~\ref{interface_ele}.

On the concrete element $T=\triangle A_1A_2A_3$, the standard Raviart-Thomas basis functions defined in (\ref{basis_standard}) can be written as 
\begin{equation*}
\boldsymbol{\lambda}_{i,T}(\boldsymbol{x})=\frac{|e_i|}{2|T|}\overrightarrow{A_i\boldsymbol{x}},\qquad i=1,2,3.
\end{equation*}
It follows from  (\ref{dof}) and (\ref{phijt_c}) that 
\begin{equation*}
\begin{aligned}
(\Pi_{h,T}&\boldsymbol{\omega})(\boldsymbol{x})\cdot\textbf{t}_h=N_{1,T}(\boldsymbol{\omega})\boldsymbol{\lambda}_{1,T}(\boldsymbol{x})\cdot\textbf{t}_h+N_{2,T}(\boldsymbol{\omega})\boldsymbol{\lambda}_{2,T}(\boldsymbol{x})\cdot\textbf{t}_h\\
&=\frac{|A_3D|}{|e_1|}\textbf{t}_h\cdot\textbf{n}_{1,T}\frac{|e_1|}{2|T|}\overrightarrow{A_1\boldsymbol{x}}\cdot\textbf{t}_h + \frac{|A_3E|}{|e_2|}\textbf{t}_h\cdot\textbf{n}_{2,T}\frac{|e_2|}{2|T|}\overrightarrow{A_2\boldsymbol{x}}\cdot\textbf{t}_h\\
&=\frac{|A_3D|}{|e_1|}(R_{\frac{\pi}{2}}\textbf{t}_h)\cdot (R_{\frac{\pi}{2}}\textbf{n}_{1,T})\frac{|e_1|}{2|T|}\overrightarrow{A_1\boldsymbol{x}}\cdot\textbf{t}_h + \frac{|A_3E|}{|e_2|}(R_{\frac{\pi}{2}}\textbf{t}_h)\cdot (R_{\frac{\pi}{2}}\textbf{n}_{2,T})\frac{|e_2|}{2|T|}\overrightarrow{A_2\boldsymbol{x}}\cdot\textbf{t}_h\\
&=(2|T|)^{-1}\left((\textbf{n}_h\cdot \overrightarrow{DA_3})(\overrightarrow{A_1\boldsymbol{x}}\cdot\textbf{t}_h)  + (\textbf{n}_h\cdot \overrightarrow{A_3E})(\overrightarrow{A_2\boldsymbol{x}}\cdot\textbf{t}_h)\right),
\end{aligned}
\end{equation*}
where $R_{\frac{\pi}{2}}$ is a rotation matrix defined in (\ref{rotation}).
Using the relation $\textbf{n}_h\cdot \overrightarrow{DA_3}+\textbf{n}_h\cdot \overrightarrow{A_3E}=\textbf{n}_h\cdot \overrightarrow{DE}=0$, we further have 
\begin{equation}\label{pro_01_1}
\begin{aligned}
(\Pi_{h,T}\boldsymbol{\omega})(\boldsymbol{x})\cdot\textbf{t}_h&=(2|T|)^{-1}\left((\textbf{n}_h\cdot \overrightarrow{DA_3})(\overrightarrow{A_1\boldsymbol{x}}\cdot\textbf{t}_h)  - (\textbf{n}_h\cdot \overrightarrow{DA_3})(\overrightarrow{A_2\boldsymbol{x}}\cdot\textbf{t}_h)\right)\\
&=(2|T|)^{-1}(\textbf{n}_h\cdot \overrightarrow{DA_3})(\overrightarrow{A_1A_2}\cdot\textbf{t}_h)\\
&=\frac{|DA_3|}{|e_1|\sin\angle A_2}\left(\textbf{n}_h\cdot \frac{\overrightarrow{DA_3}}{|DA_3|}\right)\left(\frac{\overrightarrow{A_1A_2}}{|e_3|}\cdot\textbf{t}_h\right)\\
&=\frac{|DA_3|\sin\theta\cos\gamma}{|e_1|\sin\angle A_2},
\end{aligned}
\end{equation}
where $\theta$ is the angle from   $\overrightarrow{A_3A_2}$ to $\textbf{t}_h$, and $\gamma$ is the angle from $\overrightarrow{A_1A_2}$ to $\textbf{t}_h$ (see Figure~\ref{interface_ele}). 
It is easy to see that 
\begin{equation}\label{gamma_fw}
0<\theta< \angle A_1+ \angle A_2,\qquad -\angle A_2 < \gamma <\angle A_1.
\end{equation}
Since $\alpha_{max}\leq \pi/2$, using (\ref{pro_01_1}) we get the first inequality in (\ref{est01})  
\begin{equation*}
(\Pi_{h,T}\boldsymbol{\omega})(\boldsymbol{x})\cdot\textbf{t}_h\geq0.
\end{equation*}
Finally, we prove the second inequality in (\ref{est01}).
If $0<\theta\leq \angle A_2$ (see Figure~\ref{interface_ele1}), we know from (\ref{pro_01_1}) that 
\begin{equation*}
(\Pi_{h,T}\boldsymbol{\omega})(\boldsymbol{x})\cdot\textbf{t}_h\leq\frac{|DA_3|\cos\gamma}{|e_1|}\leq 1.
\end{equation*}
If $\theta> \angle A_2$, then $\gamma>0$ (see Figure~\ref{interface_ele2}).  We need a refined estimate. Let $D^\prime$ be a point on $e_1$ such that the line $D^\prime A_1$ is parallel to the line $DE$. Then we have 
\begin{equation*}
\frac{|DA_3|}{|e_1|} \leq \frac{|D^\prime A_3|}{|e_1|}=1-\frac{|D^\prime A_2|}{|e_1|}=1-\frac{|e_3|\sin\gamma}{|e_1|\sin\theta}=1-\frac{\sin\angle A_3\sin\gamma}{\sin\angle A_1\sin\theta},
\end{equation*}
which together with (\ref{pro_01_1}) and the facts $\theta=\gamma+\angle A_2$ and $\angle A_1+\angle A_2+\angle A_3=\pi$ yields 
\begin{equation*}
(\Pi_{h,T}\boldsymbol{\omega})(\boldsymbol{x})\cdot\textbf{t}_h\leq \left(1-\frac{\sin (\angle A_1+\angle A_2)\sin\gamma}{\sin\angle A_1\sin(\gamma+\angle A_2)}\right)\frac{\sin(\gamma+\angle A_2)\cos\gamma}{\sin\angle A_2}.
\end{equation*}
By a direct calculation, we obtain 
\begin{equation*}
(\Pi_{h,T}\boldsymbol{\omega})(\boldsymbol{x})\cdot\textbf{t}_h\leq \cos^2\gamma-\frac{\sin\gamma\cos\gamma\cos\angle A_1}{\sin\angle A_1}\leq 1,
\end{equation*}
where the facts  $0<\gamma<\pi/2$ and $ 0<\angle A_1\leq \pi/2$ are used in the last inequality. This completes the proof of the lemma.
\end{proof}

\begin{lemma}\label{lem_basis}
Let $T\in\mathcal{T}_h^\Gamma$ be an interface triangle satisfying the maximum angle condition $\alpha_{max}\leq \pi/2$.  Then the function $\boldsymbol{\phi}\in \mathcal{IRT}(T)$ is uniquely determined by $N_{i,T}(\boldsymbol{\phi})$, $i=1,2,3$. Furthermore, we have the following explicit formula 
 \begin{equation}\label{ire_express}
 \boldsymbol{\phi}=\sum_{i=1}^3N_{i,T}(\boldsymbol{\phi})\boldsymbol{\lambda}_{i,T}+\frac{(\beta_T^-/\beta_T^+-1)\sum_{i=1}^3N_{i,T}(\boldsymbol{\phi})\boldsymbol{\lambda}_{i,T}(\boldsymbol{x}_T)\cdot \textbf{t}_h}{1+(\beta_T^-/\beta_T^+-1)(\Pi_{h,T}\boldsymbol{\omega})(\boldsymbol{x}_T)\cdot\textbf{t}_h}(\boldsymbol{\omega}-\Pi_{h,T}\boldsymbol{\omega}),
 \end{equation}
where $\boldsymbol{\lambda}_{i,T}$ and $\boldsymbol{\omega}$ are defined in (\ref{basis_standard}) and  (\ref{phijt_c}), respectively.
\end{lemma}
\begin{proof}
 From Lemma~(\ref{lem_01}), we have
 \begin{equation}\label{unique_fenmu}
 1+(\beta_T^-/\beta_T^+-1)(\Pi_{h,T}\boldsymbol{\omega})(\boldsymbol{x}_T)\cdot\textbf{t}_h\geq 
 \left\{\begin{aligned} 
 &1 \qquad &&\mbox{ if } \beta_T^-/\beta_T^+\geq 1,\\
 &\beta_T^-/\beta_T^+&&\mbox{ if } 0<\beta_T^-/\beta_T^+<1,
 \end{aligned}
 \right.
 \end{equation}
 which implies the equation  (\ref{eq_lambda1}) has a unique solution 
 \begin{equation*}
 \mu=\frac{(\beta_T^-/\beta_T^+-1)\boldsymbol{\phi}^0(\boldsymbol{x}_T)\cdot \textbf{t}_h}{1+(\beta_T^-/\beta_T^+-1)(\Pi_{h,T}\boldsymbol{\omega})(\boldsymbol{x}_T)\cdot\textbf{t}_h}.
 \end{equation*}
 Substituting the above identity  into (\ref{decomp1}) yields 
 \begin{equation*}
 \boldsymbol{\phi}=\boldsymbol{\phi}^0+\frac{(\beta_T^-/\beta_T^+-1)(\boldsymbol{\phi}^0(\boldsymbol{x}_T)\cdot \textbf{t}_h)\boldsymbol{\phi}^{J_t}}{1+(\beta_T^-/\beta_T^+-1)(\Pi_{h,T}\boldsymbol{\omega})(\boldsymbol{x}_T)\cdot\textbf{t}_h}.
 \end{equation*}
The result (\ref{ire_express})  is then obtained by using (\ref{expressphi0}) and (\ref{phijt_c}).
If $N_{i,T}(\boldsymbol{\phi})=0$, $i=1,2,3$, it is easy to see from (\ref{ire_express})  that $\boldsymbol{\phi}=0$. Hence, the function $\boldsymbol{\phi}\in \mathcal{IRT}(T)$ is uniquely determined by $N_{i,T}(\boldsymbol{\phi})$, $i=1,2,3$, which completes the proof.
\end{proof}

\begin{remark}
If $\beta_T^+=\beta_T^-$,  we can see from (\ref{ire_express}) that the IFE shape function space $\mathcal{IRT}(T)$ is the same as the traditional Raviart-Thomas shape function space $\mathcal{RT}(T)$. Therefore, when the interface disappears, i.e., $[\widetilde{\beta}]_{\Gamma}(\boldsymbol{x})=0$ for all $x\in\Gamma$, the IFE method (\ref{mix_IFE}) becomes the traditional Raviart-Thomas mixed finite element method.
\end{remark}

\subsection{Estimates of IFE basis functions}

On each interface element $T\in\mathcal{T}_h^\Gamma$, define IFE  basis functions by
\begin{equation}\label{IFE_basis}
\boldsymbol{\phi}_{i,T}\in \mathcal{IRT}(T),\quad  N_{j,T}(\boldsymbol{\phi}_{i,T})=\delta_{ij},~~ i,j=1,2,3, 
\end{equation}
which can be computed by  (\ref{ire_express}) easily in practical implementation.
 
It is well-known that the traditional basis functions defined in (\ref{basis_standard}) satisfy 
\begin{equation}\label{basis_stad_est}
\|\boldsymbol{\lambda}_{i,T}\|_{L^\infty(T)}\leq C,\quad i=1,2,3\qquad\forall T\in\mathcal{T}_h,
\end{equation}
where the constant  depends only on the shape regularity parameter $\varrho$.
The following lemma shows that the IFE basis functions also have similar estimates.
\begin{lemma}\label{lem_est_IFE}
There exists a constant $C$,   depending only on the shape regularity parameter $\varrho$, such that
\begin{equation}\label{est_IFE1}
\|\boldsymbol{\phi}_{i,T}\|_{L^\infty(T)}\leq C\max\{\beta_T^+/\beta_T^-,\beta_T^-/\beta_T^+\}, \quad i=1,2,3\qquad \forall T\in\mathcal{T}_h^\Gamma.
\end{equation}
Furthermore, if define $\boldsymbol{\phi}^s_{i,T}\in\mathcal{RT}(T)$ such that  $\boldsymbol{\phi}^s_{i,T}=\boldsymbol{\phi}_{i,T}|_{T_h^s}$, then it holds 
\begin{equation}\label{est_IFE2}
\|\boldsymbol{\phi}^s_{i,T}\|_{L^\infty(T)}\leq C\max\{\beta_T^+/\beta_T^-,\beta_T^-/\beta_T^+\},\quad s=+,-,  \quad i=1,2,3\qquad \forall T\in\mathcal{T}_h^\Gamma.
\end{equation}
\end{lemma}
\begin{proof}
We just need to prove the estimate (\ref{est_IFE2}) because (\ref{est_IFE1}) is  a direct consequence of (\ref{est_IFE2}).
From (\ref{IFE_basis}) and (\ref{ire_express}),  the IFE basis functions can be expressed explicitly
\begin{equation*}
\boldsymbol{\phi}_{i,T}^s=\boldsymbol{\lambda}_{i,T}+\frac{(\beta_T^-/\beta_T^+-1)\boldsymbol{\lambda}_{i,T}(\boldsymbol{x}_T)\cdot \textbf{t}_h}{1+(\beta_T^-/\beta_T^+-1)(\Pi_{h,T}\boldsymbol{\omega})(\boldsymbol{x}_T)\cdot\textbf{t}_h}(\boldsymbol{\omega}^s-\Pi_{h,T}\boldsymbol{\omega}),~~~s=+,-,
\end{equation*}
which together with (\ref{phijt_c}) and (\ref{basis_stad_est}) implies 
\begin{equation*}
\|\boldsymbol{\phi}_{i,T}^s\|_{L^\infty(T)}\leq C+\frac{C\left|\beta_T^-/\beta_T^+-1\right|}{\left|1+(\beta_T^-/\beta_T^+-1)(\Pi_{h,T}\boldsymbol{\omega})(\boldsymbol{x}_T)\cdot\textbf{t}_h\right|},~~~s=+,-.
\end{equation*}
Applying the estimate (\ref{unique_fenmu}) to the above inequality, we can derive
\begin{equation*}
\begin{aligned}
&\|\boldsymbol{\phi}_{i,T}^s\|_{L^\infty(T)}\leq C\left(1+\beta_T^-/\beta_T^+-1\right)\leq C\beta_T^-/\beta_T^+\quad &&\mbox{ if } \beta_T^-/\beta_T^+\geq 1,\\
&\|\boldsymbol{\phi}_{i,T}^s\|_{L^\infty(T)}\leq C\left(1+\frac{1-\beta_T^-/\beta_T^+}{\beta_T^-/\beta_T^+}\right)\leq C\beta_T^+/\beta_T^-\quad &&\mbox{ if } 0<\beta_T^-/\beta_T^+< 1,
\end{aligned}
\end{equation*}
which completes the proof.
\end{proof}
\subsection{The commuting diagram}
Define a local $L^2$ projection $P^0_{T}: L^2(T)\rightarrow \mathbb{R}^1$ such that
\begin{equation}\label{def_L2_pro}
P^0_{T} v=\frac{1}{|T|}\int_{T}vd\boldsymbol{x}
\end{equation}
and a global $L^2$ projection  $P^0_h: L^2(\Omega) \rightarrow M_h$ such that 
$(P^0_h v)|_T=P^0_{T} v$ for all  $T\in\mathcal{T}_h$.
It is well-known that the following commuting property holds for the standard  Raviart-Thomas interpolation operator $\Pi_h$ defined in (\ref{inter_oper_def})
\begin{equation*}
\nabla \cdot \Pi_h \boldsymbol{q}=P_h^0\nabla\cdot\boldsymbol{q}\qquad \forall \boldsymbol{q}\in W(\Omega).
\end{equation*}
Next we show that the commuting property also holds for the IFE interpolation operator $\Pi^{IFE}_h$.
It follows from (\ref{def_L2_pro}), (\ref{dof})  and (\ref{local_IFE_inter}) that
\begin{equation}\label{pro_commu}
\begin{aligned}
\int_{T}\nabla \cdot \Pi^{IFE}_h \boldsymbol{q}-P_h^0\nabla\cdot\boldsymbol{q}d\boldsymbol{x}&=\int_{T}\nabla \cdot \Pi^{IFE}_h \boldsymbol{q}-\nabla\cdot\boldsymbol{q}d\boldsymbol{x}=\sum_{i=1}^3\int_{e_i} (\Pi^{IFE}_h \boldsymbol{q}-\boldsymbol{q})\cdot \textbf{n}_{i,T}ds\\
&=\sum_{i=1}^3|e_i|N_{i,T}(\Pi^{IFE}_h \boldsymbol{q}-\boldsymbol{q})=0\qquad \forall T\in\mathcal{T}_h^\Gamma.
\end{aligned}
\end{equation}
Moreover, we know from (\ref{IFE_con3}) that 
$\nabla \cdot (\Pi^{IFE}_h \boldsymbol{q})|_{T_h^+}=\nabla \cdot (\Pi^{IFE}_h \boldsymbol{q})|_{T_h^-}$.
In other words, $(\nabla \cdot \Pi^{IFE}_h\boldsymbol{q})|_{T}$ is a constant. Thus,  (\ref{pro_commu}) implies  
\begin{equation*}
\nabla_h \cdot \Pi^{IFE}_h \boldsymbol{q}-P_h^0\nabla\cdot\boldsymbol{q}=0\qquad \forall \boldsymbol{q}\in W(\Omega).
\end{equation*}
We summarize these commuting results in the following  diagram:
\begin{equation}\label{commuting}
\begin{array}{lll}
 W(\Omega)&\xrightarrow{~~~\nabla \cdot~~~} & L^2(\Omega)   \\
 ~~~\left\downarrow \Pi_h\begin{aligned}
 ~\\
 ~
 \end{aligned}\right. &   &   ~~\left\downarrow P^0_h\begin{aligned}
 ~\\
 ~
 \end{aligned}\right.\\
 \mathcal{RT}(\mathcal{T}_h) &\xrightarrow{~~~\nabla \cdot~~~} &~~M_h
\end{array}
\qquad
\begin{array}{lll}
 W(\Omega)&\xrightarrow{~~~\nabla \cdot~~~} & L^2(\Omega)   \\
 ~~~\left\downarrow \Pi^{IFE}_h\begin{aligned}
 ~\\
 ~
 \end{aligned}\right. &   &   ~~\left\downarrow P^0_h\begin{aligned}
 ~\\
 ~
 \end{aligned}\right.\\
 \mathcal{IRT}(\mathcal{T}_h) &\xrightarrow{~~~\nabla_h \cdot~~~} &~~M_h.
\end{array}
\end{equation}

\subsection{Approximation capabilities of the IFE space}
Denote $\mbox{dist}(x,\Gamma)$ as  the distance between a point $x$ and the interface $\Gamma$, and  $U(\Gamma,\delta)=\{x\in\mathbb{R}^2: \mbox{dist}(x,\Gamma)< \delta\}$  as the neighborhood of $\Gamma$ of thickness $\delta$. 
Define a signed distance function near the interface as
\begin{equation*}
\rho(\boldsymbol{x})=\left\{
\begin{aligned}
&\mbox{dist}(x,\Gamma)\quad&& \mbox{ if } x\in\Omega^+\cap U(\Gamma,\delta_0),\\
&-\mbox{dist}(x,\Gamma)\quad&& \mbox{ if } x\in\Omega^-\cap U(\Gamma,\delta_0).
\end{aligned}\right.
\end{equation*}
We also define the meshsize of $\mathcal{T}_h^\Gamma$ by
\begin{equation}\label{hgamma}
h_\Gamma:=\max_{T\in\mathcal{T}_h^\Gamma}h_T.
\end{equation}
It is obvious that $h_\Gamma\leq h$ and $\bigcup_{T\in\mathcal{T}_h^\Gamma} T\subset U(\Gamma,h_\Gamma)$.
\begin{assumption}\label{assumption_delta}
There exists a constant $\delta_0>0$ such that the signed distance function $\rho(\boldsymbol{x})$ is well-defined in $U(\Gamma,\delta_0)$ with $\rho(\boldsymbol{x})\in C^2(U(\Gamma,\delta_0))$. We also assume $h_\Gamma<\delta_0$ 
so that  $T\subset U(\Gamma,\delta_0)$ for all interface elements $T\in\mathcal{T}_h^\Gamma$. 
\end{assumption}
The assumption is reasonable since the interface $\Gamma$ is $C^2$-smooth. 
Now the unit normal and tangent vectors of the interface can be evaluated as 
\begin{equation}\label{n_and_t}
   \textbf{n}(\boldsymbol{x})=\nabla \rho,~~~~ \textbf{t}(\boldsymbol{x}) =\left(-\frac{\partial\rho}{\partial x_2},\frac{\partial\rho}{\partial x_1}\right)^T.
\end{equation}
We note that these functions $\textbf{n}(\boldsymbol{x})$ and $\textbf{t}(\boldsymbol{x})$ are well-defined in the region $U(\Gamma,\delta_0)$.
We also view the functions $\textbf{n}_h(\boldsymbol{x})$ and $\textbf{t}_h(\boldsymbol{x})$ as piecewise constant vectors defined on interface elements.
On each interface element $T\in\mathcal{T}_h^\Gamma$, since $\Gamma$ is in  $C^2$, by  Rolle's Theorem, there exists at least one point $\boldsymbol{x}^*\in\Gamma\cap T$, see Figure~\ref{interface_ele},  such that
\begin{equation}\label{rolle}
\textbf{n}(\boldsymbol{x}^*)=\textbf{n}_h(\boldsymbol{x}^*) \quad \mbox{ and }\quad \textbf{t}(\boldsymbol{x}^*)=\textbf{t}_h(\boldsymbol{x}^*).
\end{equation}
Since $\rho(\boldsymbol{x})\in C^2(U(\Gamma,\delta_0))$, we have
\begin{equation}\label{nt_smooth}
\textbf{n}(\boldsymbol{x})\in \left(C^1(\overline{T})\right)^2\quad \mbox{ and }\quad\textbf{t}(\boldsymbol{x})\in \left(C^1(\overline{T})\right)^2\quad \forall T\in\mathcal{T}_h^\Gamma.
\end{equation}
Using  Taylor's expansion  at $x^*$, we further have
\begin{equation}\label{error_nt}
\|\textbf{n}-\textbf{n}_h\|_{L^\infty(T)}\leq Ch\quad \mbox{ and }\quad\|\textbf{t}-\textbf{t}_h\|_{L^\infty(T)}\leq Ch\quad \forall T\in\mathcal{T}_h^\Gamma.
\end{equation}
The following lemma presents a $\delta$-strip argument that will be used for the error estimate in the region near the interface (see Lemma 2.1 in \cite{Li2010Optimal}).
\begin{lemma}\label{lem_strip}
Let $\delta$ be sufficiently small. The constant $\delta_0$ is fixed and satisfies Assumption~\ref{assumption_delta}.  Then it holds  for any $v\in H^1(\Omega)$ that 
\begin{equation}\label{strip1}
\|v\|_{L^2(U(\Gamma,\delta))}\leq C\sqrt{\delta}\,  \| v\|_{H^1(U(\Gamma,\delta_0))}.
\end{equation}
Furthermore,  if $v|_{\Gamma}=0$, then there holds
\begin{equation}\label{strip2}
\|v\|_{L^2(U(\Gamma,\delta))}\leq C\delta \, \|\nabla v\|_{L^2(U(\Gamma,\delta))}.
\end{equation}
\end{lemma}
Given two scalar or vector-valued functions $v^+$ and $v^-$ on an element $T$,  define
 \begin{equation*}
 [\![v^\pm]\!](\boldsymbol{x}):=v^+(\boldsymbol{x})-v^-(\boldsymbol{x})\qquad \forall x\in T.
 \end{equation*}
We note that the difference between $ [\![\cdot]\!](\boldsymbol{x})$ and $[\cdot]_\Gamma(\boldsymbol{x})$ is the range of $\boldsymbol{x}$.
On each interface element $T\in\mathcal{T}_h^\Gamma$, define the following auxiliary functions  
\begin{equation}\label{def_auxi0}
\boldsymbol{\Psi}^s_T\in \mathcal{RT}(T),~ \boldsymbol{\Upsilon}^s_T\in \mathcal{RT}(T), ~\boldsymbol{\Theta}^s_T\in \mathcal{RT}(T),~ s=+,-
\end{equation}
and 
\begin{equation}\label{def_auxi1}
\boldsymbol{\Psi}_T|_{T_h^s}=\boldsymbol{\Psi}^s_T,~\boldsymbol{\Upsilon}_T|_{T_h^s}=\boldsymbol{\Upsilon}^s_T, ~\boldsymbol{\Theta}_T|_{T_h^s}=\boldsymbol{\Theta}^s_T,~ s=+,-,
\end{equation}
such that
\begin{equation}\label{def_auxi_condi}
\begin{aligned}
&[\![\boldsymbol{\Psi}_T^\pm\cdot \textbf{n}_h]\!](\boldsymbol{x}_T)=1,\quad [\![\beta_T^\pm\boldsymbol{\Psi}_T^\pm\cdot \textbf{t}_h]\!](\boldsymbol{x}_T)=0,\quad  [\![\nabla \cdot \boldsymbol{\Psi}_T^\pm]\!]=0,\quad N_{i,T}(\boldsymbol{\Psi}_T)=0,\\
&[\![\boldsymbol{\Upsilon}_T^\pm\cdot \textbf{n}_h]\!](\boldsymbol{x}_T)=0,\quad [\![\beta_T^\pm\boldsymbol{\Upsilon}_T^\pm\cdot \textbf{t}_h]\!](\boldsymbol{x}_T)=1,\quad  [\![\nabla \cdot \boldsymbol{\Upsilon}_T^\pm]\!]=0,\quad N_{i,T}(\boldsymbol{\Upsilon}_T)=0,\\
&[\![\boldsymbol{\Theta}_T^\pm\cdot \textbf{n}_h]\!](\boldsymbol{x}_T)=0,\quad [\![\beta_T^\pm\boldsymbol{\Theta}_T^\pm\cdot \textbf{t}_h]\!](\boldsymbol{x}_T)=0,\quad [\![ \nabla \cdot \boldsymbol{\Theta}_T^\pm]\!]=1,\quad N_{i,T}(\boldsymbol{\Theta}_T)=0,\\
\end{aligned}
\end{equation}
where $\boldsymbol{x}_T$ and $\beta_T^\pm$ are the same as that in (\ref{IFE_con2}).

\begin{lemma}\label{lem_esti_auxi}
On each interface element $T\in\mathcal{T}_h^\Gamma$, the functions $\boldsymbol{\Psi}^\pm_T$, $\boldsymbol{\Upsilon}^\pm_T$ and $\boldsymbol{\Theta}^\pm_T$ defined in (\ref{def_auxi0})-(\ref{def_auxi_condi}) exist and satisfy 
\begin{equation}\label{auxi_est0}
\|\boldsymbol{\Psi}_T^s\|_{L^2(T)}\leq C\frac{\beta_{max}}{\beta_{min}}h_T,~ \|\boldsymbol{\Upsilon}_T^s\|_{L^2(T)}\leq C\frac{1}{\beta_{min}}h_T, ~\|\boldsymbol{\Theta}_T^s\|_{L^2(T)}\leq C\frac{\beta_{max}}{\beta_{min}}h_T^2,~s=+,-, 
\end{equation}
where $\beta_{max}$ and $\beta_{min}$  are define in (\ref{betamaxmin}), and the constant $C$  depends only on the shape regularity parameter $\varrho$.
\end{lemma}
\begin{proof}
We construct $\boldsymbol{\Psi}_T$ as follows,
\begin{equation}\label{Psi_cos}
\boldsymbol{\Psi}_T=\boldsymbol{z}-\Pi_{h,T}^{IFE}\boldsymbol{z},~~~\boldsymbol{z}=\left\{
\begin{aligned}
&\boldsymbol{z}^+=\textbf{n}_h\qquad&&\mbox{ in } T_h^+,\\
&\boldsymbol{z}^-=\mathbf{0}&&\mbox{ in } T_h^-.
\end{aligned}\right.
\end{equation}
It is easy to verify that $\boldsymbol{\Psi}_T^+=\boldsymbol{\Psi}_T|_{T_h^+}$ and $\boldsymbol{\Psi}_T^-=\boldsymbol{\Psi}_T^-|_{T_h^-}$ exist and satisfy (\ref{def_auxi0})-(\ref{def_auxi_condi}). From (\ref{Psi_cos}) and (\ref{dof}), we have $|N_{i,T}(\boldsymbol{z})|\leq C$, $i=1,2,3$. Thus,  by  Lemma~\ref{lem_est_IFE}, it holds
\begin{equation*}
\|\Pi_{h,T}^{IFE}\boldsymbol{z}\|_{L^\infty(T)}\leq \sum_{i=1}^{3}|N_{i,T}(\boldsymbol{z})|\|\boldsymbol{\phi}_{i,T}\|_{L^\infty(T)}\leq C\max\{\beta_T^+/\beta_T^-,\beta_T^-/\beta_T^+\}\leq C\frac{\beta_{max}}{\beta_{min}},
\end{equation*}
which together with (\ref{Psi_cos}) yields the first estimate in (\ref{auxi_est0}) 
\begin{equation*}
\|\boldsymbol{\Psi}^s_T\|_{L^2(T)} \leq Ch_T\|\boldsymbol{\Psi}^s_T\|_{L^\infty(T)}=Ch_T \|\boldsymbol{z}^s-\Pi_{h,T}^{IFE}\boldsymbol{z}\|_{L^\infty(T)}\leq C\frac{\beta_{max}}{\beta_{min}}h_T,~~s=+,-.
\end{equation*}
Similarly, we construct $\boldsymbol{\Upsilon}_T$ as
\begin{equation}\label{Upsilon_cos}
\boldsymbol{\Upsilon}_T=\boldsymbol{z}-\Pi_{h,T}^{IFE}\boldsymbol{z},~~~\boldsymbol{z}=\left\{
\begin{aligned}
&\boldsymbol{z}^+\qquad&&\mbox{ in } T_h^+,\\
&\boldsymbol{z}^-&&\mbox{ in } T_h^-.
\end{aligned}\right.
\end{equation}
where
\begin{equation}\label{Upsilon_cos2}
\begin{aligned}
&\boldsymbol{z}^+=\textbf{t}_h/\beta_T^+,\quad &&\boldsymbol{z}^-=\mathbf{0}\quad &&\mbox{ if }\quad \beta_T^+>\beta_T^-,\\
&\boldsymbol{z}^+=\mathbf{0},\quad &&\boldsymbol{z}^-=-\textbf{t}_h/\beta_T^-\quad &&\mbox{ if } \quad\beta_T^+\leq\beta_T^-.\\
\end{aligned}
\end{equation}
We can also verify easily that $\boldsymbol{\Upsilon}_T^s=\boldsymbol{\Upsilon}_T|_{T_h^s}$, $s=+,-$ exist and satisfy (\ref{def_auxi0})-(\ref{def_auxi_condi}).  Using (\ref{Upsilon_cos2}), (\ref{dof}) and Lemma~\ref{lem_est_IFE}, we have
\begin{equation*}
\|\Pi_{h,T}^{IFE}\boldsymbol{z}\|_{L^\infty(T)}\leq \sum_{i=1}^{3}|N_{i,T}(\boldsymbol{z})|\|\boldsymbol{\phi}_{i,T}\|_{L^\infty(T)}\leq C\max\{1/\beta_T^-,1/\beta_T^+\}\leq \frac{C}{\beta_{min}},
\end{equation*}
which together with (\ref{Upsilon_cos}) and (\ref{Upsilon_cos2}) implies the second estimate in (\ref{auxi_est0}) 
\begin{equation*}
\|\boldsymbol{\Upsilon}^s_T\|_{L^2(T)} \leq Ch_T\|\boldsymbol{\Upsilon}^s_T\|_{L^\infty(T)}=Ch_T \|\boldsymbol{z}^s-\Pi_{h,T}^{IFE}\boldsymbol{z}\|_{L^\infty(T)}\leq C\frac{1}{\beta_{min}}h_T,~~s=+,-.
\end{equation*}
Finally, we construct 
\begin{equation}\label{Theta_cos}
\boldsymbol{\Theta}_T=\boldsymbol{z}-\Pi_{h,T}^{IFE}\boldsymbol{z},~~~\boldsymbol{z}=\left\{
\begin{aligned}
&\boldsymbol{z}^+=\frac{1}{2} (\boldsymbol{x}-\boldsymbol{x}_T)\qquad&&\mbox{ in } T_h^+,\\
&\boldsymbol{z}^-=\mathbf{0}&&\mbox{ in } T_h^-,
\end{aligned}\right.
\end{equation}
where $\boldsymbol{x}_T$ is the center of  the largest circle inscribed in the element $T$. Obviously, the functions $\boldsymbol{\Theta}_T^s=\boldsymbol{\Theta}_T|_{T_h^s}$, $s=+,-$  exist and satisfy (\ref{def_auxi0})-(\ref{def_auxi_condi}).  Now we have
\begin{equation*}
\|\Pi_{h,T}^{IFE}\boldsymbol{z}\|_{L^\infty(T)}\leq \sum_{i=1}^{3}|N_{i,T}(\boldsymbol{z})|\|\boldsymbol{\phi}_{i,T}\|_{L^\infty(T)}\leq Ch_T\max\{\beta_T^+/\beta_T^-,\beta_T^-/\beta_T^+\}\leq C\frac{\beta_{max}}{\beta_{min}}h_T.
\end{equation*}
Therefore, by (\ref{Theta_cos}), we get
\begin{equation*}
\|\boldsymbol{\Theta}^s_T\|_{L^2(T)} \leq Ch_T\|\boldsymbol{\Theta}^s_T\|_{L^\infty(T)}=Ch_T \|\boldsymbol{z}^s-\Pi_{h,T}^{IFE}\boldsymbol{z}\|_{L^\infty(T)}\leq C\frac{\beta_{max}}{\beta_{min}}h_T^2,
\end{equation*}
which competes the proof of this lemma.
\end{proof}

To show that functions in $\mathcal{IRT}(\mathcal{T}_h)$ can approximate $\boldsymbol{q}\in \widetilde{H}(\mathrm{div};\Omega)$ optimally in the $L^2$-norm,  we need to  interpolate extensions of $\boldsymbol{q}|_{\Omega^s}$, $s=+,-$. For any $\boldsymbol{q}^s\in (H^1(\Omega^s))^2$, from the standard Sobolev extension property (see \cite{Gilbargbook}),  there exist extensions $\boldsymbol{q}_E^s\in (H^1(\Omega^s_{ext}))^2$, $s=+,-$ such that
\begin{equation}\label{extension}
 \boldsymbol{q}_E^s|_{\Omega^s}=\boldsymbol{q}^s~~\mbox{ and }~~| \boldsymbol{q}_E^s|_{H^i(\Omega^s_{ext})}\leq C|\boldsymbol{q}^s|_{H^i(\Omega^s)},\quad  i=0,1,2,\quad s=+,-,
\end{equation}
where $\Omega^s_{ext}:=\Omega^s\cup U(\Gamma,\delta_0)$ with a fixed $\delta_0$ satisfying Assumption~\ref{assumption_delta}.

\begin{lemma}\label{lem_decomp}
On each interface element $T\in\mathcal{T}_h^\Gamma$, for any $\boldsymbol{q}\in \widetilde{H}^1(\mathrm{div};\Omega)$,  it holds that
\begin{equation}\label{decom}
\begin{aligned}
(\Pi_{T}\boldsymbol{q}_E^s)(\boldsymbol{x})-(\Pi_{T}^{IFE}\boldsymbol{q})^s(\boldsymbol{x})&=[\![(\Pi_{T}\boldsymbol{q}_E^\pm)\cdot\textbf{n}_h]\!](\boldsymbol{x}_T)\boldsymbol{\Psi}^s_T(\boldsymbol{x})+[\![\beta_T^\pm(\Pi_{T}\boldsymbol{q}_E^\pm)\cdot\textbf{t}_h]\!](\boldsymbol{x}_T)\boldsymbol{\Upsilon}_T^s(\boldsymbol{x})\\
&+[\![\nabla \cdot (\Pi_{T}\boldsymbol{q}_E^\pm)]\!]\boldsymbol{\Theta}_T^s(\boldsymbol{x})+\sum_{i=1}^3g_{i}\boldsymbol{\phi}_{i,T}(\boldsymbol{x})\quad\forall x\in T,
\end{aligned}
\end{equation}
where 
\begin{equation}\label{decom2}
g_{i}=\frac{1}{|e_i|}\left(\int_{e_i\cap \Omega^+} (\Pi_{T}\boldsymbol{q}_E^+-\boldsymbol{q}_E^+)\cdot \textbf{n}_{i,T}ds+\int_{e_i\cap \Omega^-} (\Pi_{T}\boldsymbol{q}_E^--\boldsymbol{q}_E^-)\cdot \textbf{n}_{i,T}ds\right).
\end{equation}
\end{lemma}
\begin{proof}
For simplicity of notation, we let 
\begin{equation}\label{def_wh_sim}
\boldsymbol{w}_h|_{T_h^s}:=\boldsymbol{w}_h^s, ~~\boldsymbol{w}_h^s:=\Pi_{T}\boldsymbol{q}_E^s-(\Pi_{T}^{IFE}\boldsymbol{q})^s,~~s=+,-,
\end{equation}
and define 
\begin{equation}\label{pro_def_zh}
\begin{aligned}
\boldsymbol{z}_h|_{T_h^s}:=\boldsymbol{z}^s_h,\quad \boldsymbol{z}^s_h(\boldsymbol{x})&:=[\![\boldsymbol{w}_h^\pm\cdot\textbf{n}_h]\!](\boldsymbol{x}_T)\boldsymbol{\Psi}^s_T(\boldsymbol{x})+[\![\beta_T^\pm\boldsymbol{w}_h^\pm\cdot\textbf{t}_h]\!](\boldsymbol{x}_T)\boldsymbol{\Upsilon}_T^s(\boldsymbol{x})\\
&\qquad~~~+[\![\nabla \cdot \boldsymbol{w}_h^\pm]\!]\boldsymbol{\Theta}_T^s(\boldsymbol{x})+\sum_{i=1}^3N_{i,T}(\boldsymbol{w}_h)\boldsymbol{\phi}_{i,T}(\boldsymbol{x}),\quad s=+,-.
\end{aligned}
\end{equation}
Next, we prove $\boldsymbol{w}_h=\boldsymbol{z}_h$. First, it is easy to verify from (\ref{IFE_shape})-(\ref{IFE_con3}) and (\ref{IFE_basis}) that the IFE basis functions $\boldsymbol{\phi}_{i,T}$, $i=1,2,3$ satisfy the following identities 
\begin{equation}\label{IFE_basis_jump_xt}
[\![\boldsymbol{\phi}_{i,T}^\pm\cdot \textbf{n}_h]\!](\boldsymbol{x}_T)=0,~ [\![\beta_T^\pm\boldsymbol{\phi}_{i,T}^\pm\cdot \textbf{t}_h]\!](\boldsymbol{x}_T)=0,~  [\![\nabla \cdot \boldsymbol{\phi}_{i,T}^\pm]\!]=0,~ N_{j,T}(\boldsymbol{\phi}_{i,T})=\delta_{ij}, ~j=1,2,3.
\end{equation}
Then, combining (\ref{pro_def_zh})-(\ref{IFE_basis_jump_xt}) and (\ref{def_auxi0})-(\ref{def_auxi_condi}) yields 
\begin{equation*}
\begin{aligned}
&[\![\boldsymbol{z}_h^\pm\cdot\textbf{n}_h]\!](\boldsymbol{x}_T)=[\![\boldsymbol{w}_h^\pm\cdot\textbf{n}_h]\!](\boldsymbol{x}_T),\qquad [\![\beta_T^\pm\boldsymbol{z}_h^\pm\cdot\textbf{t}_h]\!](\boldsymbol{x}_T)=[\![\beta_T^\pm\boldsymbol{w}_h^\pm\cdot\textbf{t}_h]\!](\boldsymbol{x}_T),\\
&[\![\nabla \cdot \boldsymbol{z}_h^\pm]\!]=[\![\nabla \cdot \boldsymbol{w}_h^\pm]\!],\qquad N_{i,T}(\boldsymbol{z}_h)=N_{i,T}(\boldsymbol{w}_h),
\end{aligned}
\end{equation*}
which implies 
\begin{equation*}
\boldsymbol{z}_h-\boldsymbol{w}_h\in \mathcal{IRT}(T),\quad N_{i,T}(\boldsymbol{z}_h-\boldsymbol{w}_h)=0, ~i=1,2,3.
\end{equation*}
By (\ref{ire_express}) we conclude $\boldsymbol{z}_h-\boldsymbol{w}_h=0$ (i.e., $\boldsymbol{z}_h=\boldsymbol{w}_h$).   We now have the following  decomposition
\begin{equation}\label{pro_decomp1}
\begin{aligned}
\boldsymbol{w}_h=\boldsymbol{z}_h&=[\![\boldsymbol{w}_h^\pm\cdot\textbf{n}_h]\!](\boldsymbol{x}_T)\boldsymbol{\Psi}^s_T(\boldsymbol{x})+[\![\beta_T^\pm\boldsymbol{w}_h^\pm\cdot\textbf{t}_h]\!](\boldsymbol{x}_T)\boldsymbol{\Upsilon}_T^s(\boldsymbol{x})\\
&\qquad~~~+[\![\nabla \cdot \boldsymbol{w}_h^\pm]\!]\boldsymbol{\Theta}_T^s(\boldsymbol{x})+\sum_{i=1}^3N_{i,T}(\boldsymbol{w}_h)\boldsymbol{\phi}_{i,T}(\boldsymbol{x}),\quad s=+,-.
\end{aligned}
\end{equation}

To prove the desired result  (\ref{decom}),  it remains to estimate the coefficients in the above identity.
Using the facts  from (\ref{IFE_con1})-(\ref{IFE_con3}) that 
\begin{equation*}
[\![(\Pi_{T}^{IFE}\boldsymbol{q})^\pm]\!](\boldsymbol{x}_T)=0,~[\![\beta_T^\pm(\Pi_{T}^{IFE}\boldsymbol{q})^\pm\cdot\textbf{t}_h]\!](\boldsymbol{x}_T)=0,~[\![\nabla \cdot (\Pi_{T}^{IFE}\boldsymbol{q})^\pm]\!]=0.
\end{equation*}
and  the definition of $\boldsymbol{w}_h$ in (\ref{def_wh_sim}), we can derive
\begin{equation}\label{pro_decomp2}
\begin{aligned}
&[\![\boldsymbol{w}_h^\pm\cdot\textbf{n}_h]\!](\boldsymbol{x}_T)=[\![\Pi_{T}\boldsymbol{q}_E^\pm-(\Pi_{T}^{IFE}\boldsymbol{q})^\pm]\!](\boldsymbol{x}_T)=[\![\Pi_{T}\boldsymbol{q}_E^\pm]\!](\boldsymbol{x}_T),\\
&[\![\beta_T^\pm\boldsymbol{w}_h^\pm\cdot\textbf{t}_h]\!](\boldsymbol{x}_T)=[\![\beta_T^\pm(\Pi_{T}\boldsymbol{q}_E^\pm-(\Pi_{T}^{IFE}\boldsymbol{q})^\pm)\cdot\textbf{t}_h]\!](\boldsymbol{x}_T)=[\![\beta_T^\pm\Pi_{T}\boldsymbol{q}_E^\pm\cdot\textbf{t}_h]\!](\boldsymbol{x}_T),\\
&[\![\nabla \cdot \boldsymbol{w}_h^\pm]\!]=[\![\nabla \cdot (\Pi_{T}\boldsymbol{q}_E^\pm-(\Pi_{T}^{IFE}\boldsymbol{q})^\pm)]\!]=[\![\nabla \cdot \Pi_{T}\boldsymbol{q}_E^\pm]\!],
\end{aligned}
\end{equation}
and
\begin{equation}\label{pro_decomp3}
\begin{aligned}
N_{i,T}(\boldsymbol{w}_h)&=\frac{1}{|e_i|}\sum_{s=+,-}\int_{e_i\cap\Omega^s}\left(\Pi_{T}\boldsymbol{q}_E^s-(\Pi_{T}^{IFE}\boldsymbol{q})^s\right)\cdot\textbf{n}_{i,T}ds\\
&=\frac{1}{|e_i|}\sum_{s=+,-}\int_{e_i\cap\Omega^s}\left(\Pi_{T}\boldsymbol{q}_E^s-\boldsymbol{q}_E^s+\boldsymbol{q}_E^s-(\Pi_{T}^{IFE}\boldsymbol{q})^s\right)\cdot\textbf{n}_{i,T}ds\\
&=\frac{1}{|e_i|}\sum_{s=+,-}\int_{e_i\cap\Omega^s}\left(\Pi_{T}\boldsymbol{q}_E^s-\boldsymbol{q}_E^s\right)\cdot\textbf{n}_{i,T}ds,\\
\end{aligned}
\end{equation}
where we have used the following property of the IFE interpolation operator $\Pi_{T}^{IFE}$  from  (\ref{local_IFE_inter})
\begin{equation*}
N_{i,T}(\boldsymbol{q}_E-\Pi_{T}^{IFE}\boldsymbol{q})=\sum_{s=+,-}\int_{e_i\cap \Omega^s}(\boldsymbol{q}_E^s-(\Pi_{T}^{IFE}\boldsymbol{q})^s)ds=0.
\end{equation*}
Substituting  (\ref{pro_decomp2}) and (\ref{pro_decomp3})  into (\ref{pro_decomp1}) completes the proof of this lemma.
\end{proof}

\begin{theorem}\label{theo_inter1ps}
For any $\boldsymbol{q}\in \widetilde{H}^1(\mathrm{div};\Omega)$, let $\boldsymbol{q}_E^s$ be defined in (\ref{extension}), then there exists a constant $C$, independent of $h$, $\beta_{max}$, $\beta_{min}$ and the interface location relative to the mesh, such that
\begin{equation*}
\sum_{T\in\mathcal{T}_h^\Gamma}\|\boldsymbol{q}_E^s-(\Pi_{T}^{IFE}\boldsymbol{q})^s\|^2_{L^2(T)}\leq C\frac{\beta^2_{max}}{\beta^2_{min}}h_\Gamma^2\sum_{s=+,-}\|\boldsymbol{q}_E^s\|^2_{H^1(U(\Gamma,h_\Gamma))},\qquad s=+,-,
\end{equation*}
where $h_\Gamma=\max_{T\in\mathcal{T}_h^\Gamma}h_T$.
\end{theorem}
\begin{proof}
On each interface element $T\in\mathcal{T}_h^\Gamma$,  by the triangle inequality,  we have
\begin{equation}\label{pro_int_1_err1}
\|\boldsymbol{q}_E^s-(\Pi_{T}^{IFE}\boldsymbol{q})^s\|^2_{L^2(T)}\leq 2\|\boldsymbol{q}_E^s-\Pi_{T}\boldsymbol{q}_E^s\|^2_{L^2(T)}+2\|\Pi_{T}\boldsymbol{q}_E^s-(\Pi_{T}^{IFE}\boldsymbol{q})^s\|^2_{L^2(T)},
\end{equation}
The estimate of the first term is standard 
\begin{equation}\label{pro_int_1_err2}
\|\boldsymbol{q}_E^s-\Pi_{T}\boldsymbol{q}_E^s\|^2_{L^2(T)}\leq Ch^2 |\boldsymbol{q}_E^s|^2_{H^1(T)},~ s=+,-.
\end{equation}
For the second term on the right-hand side of (\ref{pro_int_1_err1}), using Lemma~\ref{lem_decomp},  Lemma~\ref{lem_esti_auxi} and Lemma~\ref{lem_est_IFE}, we get 
\begin{equation}\label{pro_int_1_err3}
\begin{aligned}
\|\Pi_{T}\boldsymbol{q}_E^s&-(\Pi_{T}^{IFE}\boldsymbol{q})^s\|^2_{L^2(T)}\\
&\leq C([\![(\Pi_{T}\boldsymbol{q}_E^\pm)\cdot\textbf{n}_h]\!]^2(\boldsymbol{x}_T)\|\boldsymbol{\Psi}^s_T\|^2_{L^2(T)}+[\![\beta_T^\pm(\Pi_{T}\boldsymbol{q}_E^\pm)\cdot\textbf{t}_h]\!]^2(\boldsymbol{x}_T)\|\boldsymbol{\Upsilon}_T^s\|^2_{L^2(T)}\\
&\qquad\qquad+[\![\nabla \cdot (\Pi_{T}\boldsymbol{q}_E^\pm)]\!]^2\|\boldsymbol{\Theta}_T^s\|^2_{L^2(T)}+\sum_{i=1}^3g^2_{i}\|\boldsymbol{\phi}^s_{i,T}\|^2_{L^2(T)})\\
&\leq C\frac{\beta^2_{max}}{\beta^2_{min}}\left(h^2_T[\![(\Pi_{T}\boldsymbol{q}_E^\pm)\cdot\textbf{n}_h]\!]^2(\boldsymbol{x}_T)+h_T^4[\![\nabla \cdot (\Pi_{T}\boldsymbol{q}_E^\pm)]\!]^2+h_T^2\sum_{i=1}^3g^2_i\right)\\
&\qquad\qquad+C\frac{1}{\beta^2_{min}}h^2_T[\![\beta_T^\pm(\Pi_{T}\boldsymbol{q}_E^\pm)\cdot\textbf{t}_h]\!]^2(\boldsymbol{x}_T),
\end{aligned}
\end{equation}
where $g_i$ is defined in (\ref{decom2}). Next, we estimate these terms on the right-hand side of (\ref{pro_int_1_err3})  one by one.
By the standard inverse inequality and the estimate (\ref{error_nt}), we find
\begin{equation}\label{pro_int_1_err4}
\begin{aligned}
&[\![(\Pi_{T}\boldsymbol{q}_E^\pm)\cdot\textbf{n}_h]\!]^2(\boldsymbol{x}_T)=\|[\![(\Pi_{T}\boldsymbol{q}_E^\pm)\cdot\textbf{n}_h]\!]\|_{L^\infty(T)}^2\leq Ch_T^{-2}\|[\![(\Pi_{T}\boldsymbol{q}_E^\pm)\cdot\textbf{n}_h]\!]\|_{L^2(T)}^2\\
&\qquad\leq Ch_T^{-2}\left(\|[\![(\Pi_{T}\boldsymbol{q}_E^\pm-\boldsymbol{q}_E^\pm)\cdot\textbf{n}_h +\boldsymbol{q}_E^\pm\cdot (\textbf{n}_h-\textbf{n}+\textbf{n})]\!]\|_{L^2(T)}^2 \right)\\
&\qquad\leq Ch_T^{-2}\left(\|[\![(\Pi_{T}\boldsymbol{q}_E^\pm-\boldsymbol{q}_E^\pm)]\!]\cdot\textbf{n}_h\|^2_{L^2(T)}+\|[\![\boldsymbol{q}_E^\pm]\!]\cdot (\textbf{n}_h-\textbf{n})\|^2_{L^2(T)}+\|[\![\boldsymbol{q}_E^\pm\cdot\textbf{n}]\!]\|^2_{L^2(T)} \right)\\
&\qquad\leq C\sum_{s=+,-}\left(|\boldsymbol{q}^s_E|^2_{H^1(T)}+\|\boldsymbol{q}^s_E\|^2_{L^2(T)}\right)+Ch_T^{-2}\|[\![\boldsymbol{q}_E^\pm\cdot\textbf{n}]\!]\|^2_{L^2(T)}.
\end{aligned}
\end{equation}
Similarly, 
\begin{equation}\label{pro_int_1_err5}
\begin{aligned}
&[\![\beta_T^\pm(\Pi_{T}\boldsymbol{q}_E^\pm)\cdot\textbf{t}_h]\!]^2(\boldsymbol{x}_T)=\|[\![\beta_T^\pm(\Pi_{T}\boldsymbol{q}_E^\pm)\cdot\textbf{t}_h]\!]\|_{L^\infty(T)}^2\leq Ch_T^{-2}\|[\![\beta_T^\pm(\Pi_{T}\boldsymbol{q}_E^\pm)\cdot\textbf{t}_h]\!]\|_{L^2(T)}^2\\
&\quad\leq Ch_T^{-2}\left(\|[\![\beta_T^\pm(\Pi_{T}\boldsymbol{q}_E^\pm-\boldsymbol{q}_E^\pm)\cdot\textbf{t}_h +\beta_T^\pm\boldsymbol{q}_E^\pm\cdot (\textbf{t}_h-\textbf{t}+\textbf{t})]\!]\|_{L^2(T)}^2 \right)\\
&\quad\leq Ch_T^{-2}\left(\|[\![\beta_T^\pm(\Pi_{T}\boldsymbol{q}_E^\pm-\boldsymbol{q}_E^\pm)]\!]\cdot\textbf{t}_h\|^2_{L^2(T)}+\|\|[\![\beta_T^\pm\boldsymbol{q}_E^\pm]\!]\cdot (\textbf{t}_h-\textbf{t})\|^2_{L^2(T)}+\|[\![\beta_T^\pm\boldsymbol{q}_E^\pm\cdot\textbf{t}]\!]\|^2_{L^2(T)} \right)\\
&\quad\leq C\beta^2_{max}\sum_{s=+,-}\left(|\boldsymbol{q}^s_E|^2_{H^1(T)}+\|\boldsymbol{q}^s_E\|^2_{L^2(T)}\right)+Ch_T^{-2}\|[\![\beta_T^\pm\boldsymbol{q}_E^\pm\cdot\textbf{t}]\!]\|^2_{L^2(T)}.
\end{aligned}
\end{equation}
Using the commuting diagram (\ref{commuting}) and the property of the $L^2$ projection: $\|P_T^0 v\|_{L^2(T)}\leq \|v\|_{L^2(T)}$ for all $v\in L^2(T)$,  we can derive 
\begin{equation}\label{pro_int_1_err6}
\begin{aligned}
&[\![\nabla \cdot (\Pi_{T}\boldsymbol{q}_E^\pm)]\!]^2=|T|^{-1}\|[\![\nabla \cdot (\Pi_{T}\boldsymbol{q}_E^\pm)]\!]\|_{L^2(T)}^2\leq Ch_T^{-2}\sum_{s=+,-} \|\nabla \cdot \Pi_{T}\boldsymbol{q}_E^s\|^2_{L^2(T)}\\
&=Ch_T^{-2} \sum_{s=+,-} \|P^0_{T}(\nabla \cdot \boldsymbol{q}_E^s)\|^2_{L^2(T)}\leq Ch_T^{-2}\sum_{s=+,-} \|\nabla \cdot \boldsymbol{q}_E^s\|^2_{L^2(T)}\leq Ch_T^{-2}\sum_{s=+,-} |\boldsymbol{q}_E^s|^2_{H^1(T)}.
\end{aligned}
\end{equation}
By the  Cauchy-Schwarz  inequality and the standard trace inequality, the term $g_i$ defined in (\ref{decom2}) can be estimated as 
\begin{equation}\label{pro_int_1_err7}
\begin{aligned}
g^2_{i}&\leq 2|e_i|^{-2}\sum_{s=+,-}\left(\int_{e_i\cap \Omega^s} (\Pi_{T}\boldsymbol{q}_E^s-\boldsymbol{q}_E^s)\cdot \textbf{n}_{i,T}ds\right)^2\\
&\leq Ch_T^{-1}\sum_{i=+,-} \left\|(\Pi_{T}\boldsymbol{q}_E^s-\boldsymbol{q}_E^s)\cdot \textbf{n}_{i,T}\right\|^2_{L^2(e_i\cap\Omega^s)}\leq Ch_T^{-1}\sum_{i=+,-} \left\|\Pi_{T}\boldsymbol{q}_E^s-\boldsymbol{q}_E^s\right\|^2_{L^2(e_i)}\\
&\leq C\sum_{i=+,-}\left(h_T^{-2}\|\Pi_{T}\boldsymbol{q}_E^s-\boldsymbol{q}_E^s\|^2_{L^2(T)} +|\Pi_{T}\boldsymbol{q}_E^s-\boldsymbol{q}_E^s|^2_{H^1(T)} \right)\leq C\sum_{i=+,-}|\boldsymbol{q}_E^s|^2_{H^1(T)}.
\end{aligned}
\end{equation}
We now combine (\ref{pro_int_1_err1})-(\ref{pro_int_1_err7}) to obtain the error estimate on interface elements
\begin{equation*}
\begin{aligned}
\|\boldsymbol{q}_E^s-(\Pi_{T}^{IFE}\boldsymbol{q})^s&\|^2_{L^2(T)}\leq C\frac{\beta^2_{max}}{\beta^2_{min}}h_T^2\sum_{s=+,-}\|\boldsymbol{q}^s_E\|^2_{H^1(T)}\\
&+C\frac{\beta^2_{max}}{\beta^2_{min}}\|[\![\boldsymbol{q}_E^\pm\cdot\textbf{n}]\!]\|^2_{L^2(T)}+C\frac{1}{\beta^2_{min}}\|[\![\beta_T^\pm\boldsymbol{q}_E^\pm\cdot\textbf{t}]\!]\|^2_{L^2(T)}.
\end{aligned}
\end{equation*}
Summing up the estimate over all interface elements and using Assumption~\ref{assumption_delta}, we have
\begin{equation}\label{pro_int_1_err8}
\begin{aligned}
\sum_{T\in\mathcal{T}_h^\Gamma}\|\boldsymbol{q}_E^s&-(\Pi_{T}^{IFE}\boldsymbol{q})^s\|^2_{L^2(T)}\leq C\frac{\beta^2_{max}}{\beta^2_{min}}h_\Gamma^2\sum_{s=+,-}\|\boldsymbol{q}_E^s\|^2_{H^1(U(\Gamma,h_\Gamma))}\\
&+C\frac{\beta^2_{max}}{\beta^2_{min}}\|[\![\boldsymbol{q}_E^\pm\cdot\textbf{n}]\!]\|^2_{L^2(U(\Gamma,h_\Gamma))}+C\frac{1}{\beta^2_{min}}\|[\![\beta_T^\pm\boldsymbol{q}_E^\pm\cdot\textbf{t}]\!]\|^2_{L^2(U(\Gamma,h_\Gamma))}.
\end{aligned}
\end{equation}
Since the function $\boldsymbol{q}$ belongs to the space $\widetilde{H}^1(\mathrm{div};  \Omega)$ which is defined in (\ref{def_space_div_tilde}), we know 
\begin{equation*}
[\![\boldsymbol{q}_E^\pm\cdot\textbf{n}]\!](\boldsymbol{x})=0\quad \mbox{ and }\quad [\![\beta^\pm\boldsymbol{q}_E^\pm\cdot\textbf{t}]\!](\boldsymbol{x})=0 \quad \forall x\in\Gamma.
\end{equation*}
Therefore, it follows from (\ref{cond_beta}),  the estimate (\ref{strip2}) in Lemma~\ref{lem_strip} and (\ref{new_vari_deri}) that
\begin{equation}\label{pro_int_1_err9}
\begin{aligned}
&\|[\![\boldsymbol{q}_E^\pm\cdot\textbf{n}]\!]\|^2_{L^2(U(\Gamma,h_\Gamma))}\leq Ch_\Gamma^2|[\![\boldsymbol{q}_E^\pm\cdot\textbf{n}]\!]|^2_{H^1(U(\Gamma,h_\Gamma))}\leq Ch_\Gamma^2\sum_{s=+,-}|\boldsymbol{q}_E^s|^2_{H^1(U(\Gamma,h_\Gamma))}
\\
&\|[\![\beta_T^\pm\boldsymbol{q}_E^\pm\cdot\textbf{t}]\!]\|^2_{L^2(U(\Gamma,h_\Gamma))}\leq 2\|[\![\beta^\pm\boldsymbol{q}_E^\pm\cdot\textbf{t}]\!]\|^2_{L^2(U(\Gamma,h_\Gamma))}+2\|[\![(\beta^\pm-\beta_T^\pm)\boldsymbol{q}_E^\pm\cdot\textbf{t}]\!]\|^2_{L^2(U(\Gamma,h_\Gamma))}\\
&\qquad\qquad\qquad\leq Ch_\Gamma^2\left(|[\![\beta^\pm\boldsymbol{q}_E^\pm\cdot\textbf{t}]\!]|^2_{H^1(U(\Gamma,h_\Gamma))}+\sum_{s=+,-}\|\boldsymbol{q}_E^s\cdot\textbf{t}\|^2_{L^2(U(\Gamma,h_\Gamma))}\right)\\
&\qquad\qquad\qquad\leq C\beta^2_{max}h_\Gamma^2\sum_{s=+,-}\|\boldsymbol{q}_E^s\|^2_{H^1(U(\Gamma,h_\Gamma))}.
\end{aligned}
\end{equation}
%where the constant $C$ only depends on $C_\beta$ and the maximum curvature of the interface $\Gamma$. 
Finally, substituting (\ref{pro_int_1_err9}) into (\ref{pro_int_1_err8}) completes the proof of this theorem.
\end{proof}

\begin{theorem}\label{theo_inter_error}
For any $\boldsymbol{q}\in \widetilde{H}^1(\mathrm{div};\Omega)$,  there exists a constant $C$, independent of $h$, $\beta_{max}$, $\beta_{min}$ and the interface location relative to the mesh,  such that
\begin{equation*}
\|\boldsymbol{q}-(\Pi_{h}^{IFE}\boldsymbol{q})\|_{L^2(\Omega)}\leq C\frac{\beta_{max}}{\beta_{min}}h\|\boldsymbol{q}\|_{H^1(\Omega^+\cup\Omega^-)}.
\end{equation*}
\end{theorem}
\begin{proof}
On each non-interface element $T\in\mathcal{T}_h^{non}$,  the following  estimate is standard,
\begin{equation}\label{pro_int2a}
\|\boldsymbol{q}-\Pi_{h}^{IFE}\boldsymbol{q}\|^2_{L^2(T)}=\|\boldsymbol{q}-\Pi_{h}\boldsymbol{q}\|^2_{L^2(T)}\leq Ch_T^2|\boldsymbol{q}|^2_{H^1(T)}.
\end{equation}
On each interface element $T\in\mathcal{T}_h^\Gamma$, define $T^s:=T\cap \Omega^s$, $s=+,-$.  Using the relations $T=T^+\cup T^-$ and $T^s=(T^s\cap T_h^+)\cup (T^s\cap T_h^-)$, $s=+,-$, we have
\begin{equation}\label{pro_int3a}
\begin{aligned}
\|\boldsymbol{q}-\Pi_{h}^{IFE}\boldsymbol{q}\|^2_{L^2(T)}&= \sum_{s=+,-}\|\boldsymbol{q}^s-(\Pi_{h}^{IFE}\boldsymbol{q})^s\|^2_{L^2(T^s\cap T_h^s)}\\
&+\|\boldsymbol{q}^--(\Pi_{h}^{IFE}\boldsymbol{q})^+\|^2_{L^2(T^-\cap T_h^+)} +\|\boldsymbol{q}^+-(\Pi_{h}^{IFE}\boldsymbol{q})^-\|^2_{L^2(T^+ \cap T_h^-)}.
\end{aligned}
\end{equation}
It follows from  the triangle inequality that 
\begin{equation}\label{pro_int4}
\begin{aligned}
&\|\boldsymbol{q}^--(\Pi_{h}^{IFE}\boldsymbol{q})^+\|^2_{L^2(T^-\cap T_h^+)}\leq 2\|\boldsymbol{q}^--\boldsymbol{q}_E^+\|^2_{L^2(T^-\cap T_h^+)}+2\|\boldsymbol{q}_E^+-(I_h^{IFE}\boldsymbol{q})^+\|^2_{L^2(T^-\cap T_h^+)},\\
&\|\boldsymbol{q}^+-(I_h^{IFE}\boldsymbol{q})^-\|^2_{L^2(T^+\cap T_h^-)}\leq 2\|\boldsymbol{q}^+-\boldsymbol{q}_E^-\|^2_{L^2(T^+\cap T_h^-)}+ 2\|\boldsymbol{q}_E^--(I_h^{IFE}\boldsymbol{q})^-\|^2_{L^2(T^+\cap T_h^-)}.
\end{aligned}
\end{equation}
Substituting (\ref{pro_int4}) into (\ref{pro_int3a}) yields 
\begin{equation}\label{pro_int5}
\|\boldsymbol{q}-\Pi_{h}^{IFE}\boldsymbol{q}\|^2_{L^2(T)}\leq  C\sum_{s=+,-}\|\boldsymbol{q}_E^s-(I_h^{IFE}\boldsymbol{q})^s\|^2_{L^2(T)}+C\|\boldsymbol{q}_E^+-\boldsymbol{q}_E^-\|^2_{L^2(T^\triangle)},
\end{equation}
where 
\begin{equation}\label{def_t_delta}
T^\triangle:=(T^-\cap T_h^+)\cup (T^+\cap T_h^-).
\end{equation} Since the interface is $C^2$-smooth, we conclude that, there exists a constant $C$ depending only on the maximum curvature of the interface $\Gamma$, such that $\mbox{dist}(\Gamma,\Gamma_h)\leq Ch_\Gamma$, which leads to 
\begin{equation}\label{pro_int6}
\bigcup_{T\in\mathcal{T}_h^\Gamma}T^\triangle\subset U(\Gamma,Ch_\Gamma^2)\subset U(\Gamma,\delta_0).
\end{equation}
Summing up the estimates (\ref{pro_int2a}) and (\ref{pro_int5}) over all elements $T\in\mathcal{T}_h$, and using (\ref{pro_int6}) , we have
\begin{equation*}
\begin{aligned}
&\sum_{T\in\mathcal{T}_h}\|\boldsymbol{q}-I_h^{IFE}\boldsymbol{q}\|^2_{L^2(T)}\\
&\leq Ch^2\sum_{T\in\mathcal{T}_h^{non}}|\boldsymbol{q}|^2_{H^1(T)}+C\sum_{T\in\mathcal{T}_h^\Gamma}\sum_{s=+,-}\|\boldsymbol{q}_E^s-(I_h^{IFE}\boldsymbol{q})^s\|^2_{L^2(T)}+C\|\boldsymbol{q}_E^+-\boldsymbol{q}_E^-\|^2_{L^2(U(\Gamma,Ch^2_\Gamma))}.
\end{aligned}
\end{equation*}
Applying Theorem~\ref{theo_inter1ps} and the estimate (\ref{strip1}) in Lemma~\ref{lem_strip}, we further have
\begin{equation*}
\begin{aligned}
\sum_{T\in\mathcal{T}_h}\|\boldsymbol{q}-I_h^{IFE}\boldsymbol{q}\|^2_{L^2(T)}&\leq C\frac{\beta^2_{max}}{\beta^2_{min}}h^2\sum_{s=+,-}\|\boldsymbol{q}^s_E\|^2_{H^1(U(\Gamma,\delta_0))}+Ch^2\|\boldsymbol{q}_E^+-\boldsymbol{q}_E^-\|^2_{H^1(U(\Gamma,\delta_0))}\\
&\leq C\frac{\beta^2_{max}}{\beta^2_{min}}h^2\sum_{s=+,-}\|\boldsymbol{q}_E^s\|^2_{H^1(U(\Gamma,\delta_0))}+Ch^2\sum_{s=+,-}\|\boldsymbol{q}_E^s\|^2_{H^1(U(\Gamma,\delta_0))},
\end{aligned}
\end{equation*}
which together with the extension result (\ref{extension}) completes the proof.
\end{proof}

\section{A priori error estimates for the IFE method}\label{sec_error}
For the purpose of error analysis, we introduce the following mesh dependent norms for the space $\mathcal{IRT}(\mathcal{T}_h)+H(\mathrm{div};  \Omega)$,
\begin{equation*}
\begin{aligned}
\| \boldsymbol{q} \| ^2:=\|\boldsymbol{q}\|^2_{L^2(\Omega)}+\eta\sum_{e\in\mathcal{E}_h^\Gamma}\|[\boldsymbol{q}\cdot\textbf{n}_e]_e\|^2_{L^2(e)}, \quad \interleave \boldsymbol{q} \interleave^2:=\| \boldsymbol{q} \| ^2+\sum_{T\in\mathcal{T}_h}\|\nabla \cdot \boldsymbol{q}\|^2_{L^2(T)}.
\end{aligned}
\end{equation*}
\begin{lemma}\label{lem_inter_newnorm}
If $\boldsymbol{q}\in \widetilde{H}^1(\mathrm{div};  \Omega)\cap H^2(\Omega^+\cup\Omega^-)$, then there exists constant $C$, independent of $h$, the penalty parameter $\eta$  and the interface location relative to the mesh, such that
\begin{equation}\label{inter_newnorm}
\| \boldsymbol{q}-\Pi_h^{IFE}\boldsymbol{q}\|  \leq Ch(\|\boldsymbol{q}\|_{H^1(\Omega^+\cup\Omega^-)}+\eta^{1/2}\sum_{s=+,-}\|\boldsymbol{q}\|_{H^2(\Omega_{\delta_0}^+\cup \Omega_{\delta_0}^- )}),
\end{equation}
where $\Omega_{\delta_0}^s=U(\Gamma,\delta_0)\cap\Omega^s$, $s=+,-$.
\end{lemma}
\begin{proof}
First, by Theorem~\ref{theo_inter_error}, we have
\begin{equation}\label{pro_l2_nor}
\|\boldsymbol{q}-(\Pi_{h}^{IFE}\boldsymbol{q})\|_{L^2(\Omega)}\leq Ch\|\boldsymbol{q}\|_{H^1(\Omega^+\cup\Omega^-)}.
\end{equation}
By the definition of extensions (\ref{extension}), the standard trace lemma and Theorem~\ref{theo_inter1ps}, we get
\begin{equation}\label{pro_int_edge_nor}
\begin{aligned}
&\sum_{e\in\mathcal{E}_h^\Gamma}\|[(\boldsymbol{q}-\Pi_h^{IFE}\boldsymbol{q})\cdot\textbf{n}_e]_e\|^2_{L^2(e)}\leq \sum_{e\in\mathcal{E}_h^\Gamma}\sum_{s=+,-}\|[(\boldsymbol{q}_E^s-(\Pi_h^{IFE}\boldsymbol{q})^s)\cdot\textbf{n}_e]_e\|^2_{L^2(e)}\\
&\qquad\leq C\sum_{T\in\mathcal{T}_h^\Gamma}\sum_{s=+,-}\left(h_T^{-1}\|\boldsymbol{q}_E^s-(\Pi_h^{IFE}\boldsymbol{q})^s\|^2_{L^2(T)}+h_T|\boldsymbol{q}_E^s-(\Pi_h^{IFE}\boldsymbol{q})^s|^2_{H^1(T)}\right)\\
&\qquad\leq C\sum_{s=+,-}h_\Gamma\|\boldsymbol{q}_E^s\|^2_{H^1(U(\Gamma,h_\Gamma))}+C\sum_{T\in\mathcal{T}_h^\Gamma}\sum_{s=+,-}h_T|\boldsymbol{q}_E^s-(\Pi_h^{IFE}\boldsymbol{q})^s|^2_{H^1(T)}.
\end{aligned}
\end{equation}
Using the fact $(\nabla\cdot \Pi_h^{IFE}\boldsymbol{q})^+=(\nabla\cdot \Pi_h^{IFE}\boldsymbol{q})^-$ is a constant, we find
\begin{equation*}
|(\Pi_h^{IFE}\boldsymbol{q})^s|^2_{H^1(T)}=\frac{1}{2}\|\nabla\cdot(\Pi_h^{IFE}\boldsymbol{q})^s\|^2_{L^2(T)}=\frac{1}{2}\|\nabla\cdot(\Pi_h^{IFE}\boldsymbol{q})\|^2_{L^2(T)},
\end{equation*}
which together with  the property of the $L^2$ projection
$$\|P_T^0 v\|_{L^2(T)}\leq \|v\|_{L^2(T)}\quad \forall v\in L^2(T)$$ and the commuting diagram (\ref{commuting}) leads to
\begin{equation}\label{pro_inter_est_002}
|(\Pi_h^{IFE}\boldsymbol{q})^s|_{H^1(T)}=\frac{1}{2}\|\nabla \cdot  \Pi_h^{IFE}\boldsymbol{q}\|_{L^2(T)}=\frac{1}{2}\|P^0_T(\nabla\cdot\boldsymbol{q})\|_{L^2(T)}\leq \frac{1}{2}\|\nabla\cdot\boldsymbol{q}\|_{L^2(T)}.
\end{equation}
Therefore,
\begin{equation*}
|\boldsymbol{q}_E^s-(\Pi_h^{IFE}\boldsymbol{q})^s|_{H^1(T)}\leq |\boldsymbol{q}_E^s|_{H^1(T)}+|(\Pi_h^{IFE}\boldsymbol{q})^s|_{H^1(T)}\leq C|\boldsymbol{q}_E^s|_{H^1(T)},\quad s=+,-.
\end{equation*}
Substituting the above estimate into (\ref{pro_int_edge_nor}) and using (\ref{strip1}) in Lemma~\ref{lem_strip}, we obtain
\begin{equation*}
\begin{aligned}
\sum_{e\in\mathcal{E}_h^\Gamma}\|[(\boldsymbol{q}-\Pi_h^{IFE}\boldsymbol{q})\cdot\textbf{n}_e]_e\|^2_{L^2(e)}&\leq C\sum_{s=+,-}h_\Gamma\|\boldsymbol{q}_E^s\|^2_{H^1(U(\Gamma,h_\Gamma))}\leq C\sum_{s=+,-}h_\Gamma^2\|\boldsymbol{q}_E^s\|^2_{H^2(U(\Gamma,\delta_0))}.
\end{aligned}
\end{equation*}
Finally, combining (\ref{pro_l2_nor}), the above inequality and the extension result (\ref{extension}) yields the desired estimate (\ref{inter_newnorm}).
\end{proof}

It is easy to see that $b_h(\cdot,\cdot)=b(\cdot,\cdot)$ on the space $H(\mathrm{div};  \Omega)\times L^2(\Omega)$ and it is continuous:
\begin{equation}\label{bound_b}
b_h(\boldsymbol{q},v)\leq \interleave \boldsymbol{q}\interleave \|v\|_{L^2(\Omega)}\quad\forall (\boldsymbol{q}, v) \in (\mathcal{IRT}(\mathcal{T}_h)+H(\mathrm{div};  \Omega), L^2(\Omega)).
\end{equation}
Obviously, the bilinear form $a_h(\cdot,\cdot)$ is also continuous:
\begin{equation}\label{bound_a}
A_h(\boldsymbol{p},\boldsymbol{q})\leq \max\{\beta_{max},1\}\|\boldsymbol{p}\| \|\boldsymbol{q}\|\quad \forall \boldsymbol{p},\boldsymbol{q}\in \mathcal{IRT}(\mathcal{T}_h)+H(\mathrm{div};  \Omega).
\end{equation}
Let us introduce a kernel 
\begin{equation}\label{def_kernel}
\mathbb{K}_h:=\{\boldsymbol{q}_h\in\mathcal{IRT}(\mathcal{T}_h) : b_h(\boldsymbol{q}_h,v_h)=0 \quad\forall v_h\in M_h \},
\end{equation}
then we have the following coercivity  result.
\begin{lemma}\label{lem_coercivity_ah}
It holds that
\begin{equation*}
A_h(\boldsymbol{q}_h,\boldsymbol{q}_h)\geq \min\{\beta_{min},1\} \interleave \boldsymbol{q}_h \interleave^2 =\min\{ \beta_{min},1\}\|\boldsymbol{q}_h\|^2\qquad \forall\boldsymbol{q}_h\in \mathbb{K}_h.
\end{equation*}
\end{lemma}
\begin{proof}
For any $\boldsymbol{q}_h\in\mathcal{IRT}(\mathcal{T}_h)$, we know from (\ref{IFE_con3}) that  $\nabla \cdot \boldsymbol{q}_h|_{T}$ is a constant for all $T\in\mathcal{T}_h$.  Thus, $\nabla_h\cdot \boldsymbol{q}_h \in M_h$. Taking $v_h=\nabla_h \cdot \boldsymbol{q}_h$ in (\ref{def_kernel}) yields $\nabla_h  \cdot \boldsymbol{q}_h=0$. Therefore,
\begin{equation}\label{equ_kh}
\|\boldsymbol{q}_h\|^2=\interleave\boldsymbol{q}_h\interleave^2\qquad \forall\boldsymbol{q}_h\in \mathbb{K}_h,
\end{equation}
which together with 
\begin{equation*}
A_h(\boldsymbol{q}_h,\boldsymbol{q}_h)\geq\beta_{min}\|\boldsymbol{q}_h\|^2_{L^2(\Omega)}+\eta\sum_{e\in\mathcal{E}_h^\Gamma}\|[\boldsymbol{q}_h\cdot\textbf{n}_e]_e\|^2_{L^2(e)}\geq\min\{\beta_{min},1\}\|\boldsymbol{q}_h\|^2
\end{equation*}
completes the proof of this lemma.
\end{proof}

In order to prove the inf-sup condition for the bilinear form $b_h(\cdot,\cdot)$, we first state a property of the IFE operator $\Pi_h^{IFE}$ in the following lemma.
\begin{lemma}\label{lem_infsup_zhunbei}
Suppose $\mathcal{T}_h^\Gamma$ is quasi-uniform, i.e., $h_T^{-1}\leq Ch_\Gamma^{-1}$ for all $T\in\mathcal{T}_h^\Gamma$. Then there exists a constant $C_\Pi$, independent of $h$, the penalty parameter $\eta$ and the interface location relative to the mesh, such that
\begin{equation}\label{pro_inter_est_000}
\interleave \Pi_h^{IFE}\boldsymbol{q}\interleave \leq C_\Pi\frac{\beta_{max}}{\beta_{min}}\|\boldsymbol{q}\|_{H^1(
\Omega)}\qquad \forall \boldsymbol{q}\in (H^1(\Omega))^2.
\end{equation}
\end{lemma}
\begin{proof}
On each interface element $T\in\mathcal{T}_h^\Gamma$, we have
\begin{equation*}
(\Pi_h^{IFE}\boldsymbol{q})|_{T}=\sum_{i=1}^3N_{i,T}(\boldsymbol{q})\boldsymbol{\phi}_{i,T}.
\end{equation*}
It follows from  Lemma~\ref{lem_est_IFE}, the  Cauchy-Schwarz inequality and the standard trace inequality that 
\begin{equation}\label{pro_inter_est_001}
\begin{aligned}
\|\Pi_h^{IFE}&\boldsymbol{q}\|^2_{L^2(T)}\leq C\frac{\beta^2_{max}}{\beta^2_{min}}h_T^2\sum_{i=1}^3\left(\int_{e_i}\boldsymbol{q}\cdot\textbf{n}_{i,T}ds\right)^2\leq C\frac{\beta^2_{max}}{\beta^2_{min}}h_T^3\sum_{i=1}^3 \|\boldsymbol{q}\cdot\textbf{n}_{i,T}\|^2_{L^2(e_i)}\\
&\leq C\frac{\beta^2_{max}}{\beta^2_{min}}h_T^2\sum_{i=1}^3\left(h_T^{-2} \|\boldsymbol{q}\cdot\textbf{n}_{i,T}\|^2_{L^2(T)}+|\boldsymbol{q}\cdot\textbf{n}_{i,T}|^2_{H^1(T)} \right)\leq C\frac{\beta^2_{max}}{\beta^2_{min}}\|\boldsymbol{q}\|^2_{H^1(T)},
\end{aligned}
\end{equation}
which together with a similar estimate on non-interface elements leads to
\begin{equation}\label{pro_stabi_l2}
\|\Pi_h^{IFE}\boldsymbol{q}\|^2_{L^2(\Omega)}\leq C\frac{\beta^2_{max}}{\beta^2_{min}}\|\boldsymbol{q}\|^2_{H^1(\Omega)}.
\end{equation}
On the other hand, using the standard trace inequality, we get
\begin{equation}\label{pro_sta_edge_fenk}
\begin{aligned}
&\sum_{e\in\mathcal{E}_h^\Gamma}\|\Pi_h^{IFE}\boldsymbol{q}\|^2_{L^2(e)}\leq \sum_{e\in\mathcal{E}_h^\Gamma}\sum_{s=+,-}\|(\Pi_h^{IFE}\boldsymbol{q})^s\|^2_{L^2(e)}\\
&\qquad \leq C\sum_{T\in\mathcal{T}_h^\Gamma}\sum_{s=+,-}h_T^{-1}\|(\Pi_h^{IFE}\boldsymbol{q})^s\|^2_{L^2(T)} +C\sum_{T\in\mathcal{T}_h^\Gamma}\sum_{s=+,-}h_T|(\Pi_h^{IFE}\boldsymbol{q})^s|^2_{H^1(T)}.
\end{aligned}
\end{equation}
By Theorem~\ref{theo_inter1ps}, the estimate (\ref{strip1}) in Lemma~\ref{lem_strip} and the extension result (\ref{extension}), the first term on the right-hand side can be estimate as
\begin{equation}\label{pro_sta_edge_fen_1}
\begin{aligned}
&\sum_{T\in\mathcal{T}_h^\Gamma}\sum_{s=+,-}h_T^{-1}\|(\Pi_h^{IFE}\boldsymbol{q})^s\|^2_{L^2(T)}=\sum_{T\in\mathcal{T}_h^\Gamma}\sum_{s=+,-}h_T^{-1}\|\boldsymbol{q}_E^s+(\Pi_h^{IFE}\boldsymbol{q})^s-\boldsymbol{q}_E^s\|^2_{L^2(T)} \\
&\qquad\leq Ch_\Gamma^{-1}\sum_{s=+,-}\|\boldsymbol{q}_E^s\|^2_{L^2(U(\Gamma,h_\Gamma))}+Ch_\Gamma^{-1}\sum_{T\in\mathcal{T}_h^\Gamma}\sum_{s=+,-} \|\boldsymbol{q}_E^s-(\Pi_h^{IFE}\boldsymbol{q})^s\|^2_{L^2(T)}\\
&\qquad\leq C\sum_{s=+,-}\|\boldsymbol{q}_E^s\|^2_{H^1(U(\Gamma,\delta_0))}+C\frac{\beta^2_{max}}{\beta^2_{min}}h_\Gamma\sum_{s=+,-} \|\boldsymbol{q}_E^s\|^2_{H^1(U(\Gamma,h_\Gamma))}\leq C\frac{\beta^2_{max}}{\beta^2_{min}}\|\boldsymbol{q}\|^2_{H^1(\Omega)}.
\end{aligned}
\end{equation}
The estimate (\ref{pro_inter_est_002}) gives the estimate for the second term on the right-hand side of (\ref{pro_sta_edge_fenk})
\begin{equation*}
\sum_{T\in\mathcal{T}_h^\Gamma}\sum_{s=+,-}h_T|(\Pi_h^{IFE}\boldsymbol{q})^s|^2_{H^1(T)}\leq Ch_\Gamma\sum_{T\in\mathcal{T}_h^\Gamma}| \boldsymbol{q}|^2_{H^1(T)}\leq C\|\boldsymbol{q}\|^2_{H^1(\Omega)},
\end{equation*}
which together with (\ref{pro_sta_edge_fen_1}) and (\ref{pro_sta_edge_fenk}) yields 
\begin{equation}\label{pro_stabi_edge}
\sum_{e\in\mathcal{E}_h^\Gamma}\|\Pi_h^{IFE}\boldsymbol{q}\|^2_{L^2(e)}\leq C\frac{\beta^2_{max}}{\beta^2_{min}}\|\boldsymbol{q}\|^2_{H^1(\Omega)}.
\end{equation}
The estimate (\ref{pro_inter_est_002}) also implies  
\begin{equation}\label{pro_div_edge}
\sum_{T\in\mathcal{T}_h}\|\nabla \cdot\Pi_h^{IFE}\boldsymbol{q}\|^2_{L^2(T)}\leq C\sum_{T\in\mathcal{T}_h}\|\nabla \cdot\boldsymbol{q}\|_{L^2(T)}\leq C\|\boldsymbol{q}\|^2_{H^1(\Omega)}.
\end{equation}
The desired result (\ref{pro_inter_est_000}) now follows from (\ref{pro_stabi_l2}), (\ref{pro_stabi_edge}) and (\ref{pro_div_edge}).
\end{proof}

\begin{lemma}\label{lem_infsup}
Under the condition of Lemma~\ref{lem_infsup_zhunbei}, the following inf-sup condition holds for a positive constant $\eta_*$  independent of the meshsize $h$, the penalty parameter  $\eta$ and the interface location relative to the mesh
\begin{equation}\label{lbbcon}
\sup_{\boldsymbol{q}_h\in \mathcal{IRT}(\mathcal{T}_h)}\frac{b_h(\boldsymbol{q}_h,v_h)}{ \interleave \boldsymbol{q}_h \interleave }\geq \eta_*\|v_h\|_{L^2(\Omega)}\qquad \forall v_h\in M_h.
\end{equation}
\end{lemma}
\begin{proof}
Let $v_h$ be any function in $M_h\subset L^2(\Omega)$, then there is a function $\boldsymbol{w}$ satisfying (see Lemma 11.2.3 in \cite{brenner2008mathematical})
\begin{equation*}
\nabla \cdot \boldsymbol{w}=v_h~~\mbox{ and }~~ \|\boldsymbol{w}\|_{H^1(\Omega)}\leq C_{\Omega}\|v_h\|_{L^2(\Omega)}
\end{equation*}
with a constant $C_{\Omega}$ only depends on $\Omega$. It follows from the commuting diagram (\ref{commuting}) that 
\begin{equation}\label{pro_infsup_e1}
\begin{aligned}
b_h(\Pi_h^{IFE}\boldsymbol{w},v_h)&=\int_{\Omega} v_h\nabla_h \cdot \Pi_h^{IFE}\boldsymbol{w}d\boldsymbol{x}=\int_{\Omega} v_h P_h^0(\nabla \cdot \boldsymbol{w})d\boldsymbol{x}\\
&=\int_{\Omega} v_h \nabla \cdot \boldsymbol{w}d\boldsymbol{x}=\|v_h\|^2_{L^2(\Omega)}\geq \frac{1}{C_{\Omega}}\|\boldsymbol{w}\|_{H^1(\Omega)},
\end{aligned}
\end{equation}
which together with (\ref{pro_inter_est_000}) leads to
\begin{equation*}
\begin{aligned}
\sup_{\boldsymbol{q}_h\in \mathcal{IRT}(\mathcal{T}_h)}&\frac{b_h(\boldsymbol{q}_h,v_h)}{\interleave \boldsymbol{q}_h\interleave }\geq \frac{b_h(\Pi_h^{IFE}\boldsymbol{w},v_h)}{ \interleave \Pi_h^{IFE}\boldsymbol{w}\interleave }\geq \frac{\beta_{min}}{ C_\Pi \beta_{max}} \frac{b_h(\Pi_h^{IFE}\boldsymbol{w},v_h)}{\|\boldsymbol{w}\|_{H^1(\Omega)}}\geq
\frac{\beta_{min}}{ C_\Omega C_\Pi \beta_{max}}.
\end{aligned}
\end{equation*}
Thus, the inf-sup condition  (\ref{lbbcon}) is proved by choosing $\eta_*=\frac{\beta_{min}}{ C_\Omega C_\Pi \beta_{max}}$.
\end{proof}

\begin{theorem}\label{theo_error}
Suppose the condition of Lemma~\ref{lem_infsup_zhunbei} holds. The discrete method (\ref{mix_IFE}) has a unique solution $(\boldsymbol{p}_h,u_h)$. Let $(\boldsymbol{p},u)$  be the solution of problem (\ref{mix_weak}).  If $\boldsymbol{p}\in \widetilde{H}^1(\mathrm{div};  \Omega)\cap H^2(\Omega^+\cup\Omega^-)$, then there exists a constant $C$, independent of the meshsize $h$, the penalty parameter  $\eta$ and the interface location relative to the mesh, such that 
\begin{equation}\label{ulti_est}
\begin{aligned}
\|\boldsymbol{p}&-\boldsymbol{p}_h\|+\|u-u_h\|_{L^2(\Omega)}\\
&\leq Ch\left(\|\boldsymbol{p}\|_{H^1(\Omega^+\cup\Omega^-)}+\eta^{1/2}\|\boldsymbol{p}\|_{H^2(\Omega_{\delta_0}^+\cup \Omega_{\delta_0}^- )}+(1+\eta^{-1/2})\|u\|_{H^2(\Omega^+\cup\Omega^-)}\right),
\end{aligned}
\end{equation}
where $\Omega_{\delta_0}^s=U(\Gamma,\delta_0)\cap\Omega^s$, $s=+,-$.
\end{theorem}
\begin{proof}
The well-posedness of the discrete problem (\ref{mix_IFE}) is a direct consequence of (\ref{bound_b}), (\ref{bound_a}), Lemma~\ref{lem_coercivity_ah} and Lemma~\ref{lem_infsup}. 

Next, we prove the error estimate (\ref{ulti_est}). Similar to (\ref{pro_infsup_e1}), it holds 
\begin{equation*}
b_h(\Pi_h^{IFE}\boldsymbol{p},v_h)=b(\boldsymbol{p},v_h)\qquad \forall v_h\in M_h.
\end{equation*}
By subtracting the second equations in each of (\ref{mix_weak}) and (\ref{mix_IFE}), we find
\begin{equation*}
b(\boldsymbol{p},v_h)-b_h(\boldsymbol{p}_h,v_h)=0\qquad \forall v_h\in M_h.
\end{equation*}
It follows from the above two equalities that $b_h(\Pi_h^{IFE}\boldsymbol{p}-\boldsymbol{p}_h,v_h)=0$ for all  $v_h\in M_h$, which implies 
\begin{equation*}
\Pi_h^{IFE}\boldsymbol{p}-\boldsymbol{p}_h\in \mathbb{K}_h.
\end{equation*}
Then, using the triangle inequality,  Lemma~(\ref{lem_coercivity_ah}) and the continuity (\ref{bound_a}),   we have
\begin{equation}\label{pro_l201}
\begin{aligned}
\|\boldsymbol{p}&-\boldsymbol{p}_h\|\leq \|\boldsymbol{p}-\Pi_h^{IFE}\boldsymbol{p}\|+\|\Pi_h^{IFE}\boldsymbol{p}-\boldsymbol{p}_h\|\\
&\leq \|\boldsymbol{p}-\Pi_h^{IFE}\boldsymbol{p}\| +C \sup_{\boldsymbol{w}_h\in \mathbb{K}_h\backslash\{0\} }\frac{\left|A_h(\Pi_h^{IFE}\boldsymbol{p}-\boldsymbol{p}_h,\boldsymbol{w}_h)\right|}{\|\boldsymbol{w}_h\|}\\
&=  \|\boldsymbol{p}-\Pi_h^{IFE}\boldsymbol{p}\|+C\sup_{\boldsymbol{w}_h\in \mathbb{K}_h\backslash\{0\} }\frac{\left|A_h(\Pi_h^{IFE}\boldsymbol{p}-\boldsymbol{p},\boldsymbol{w}_h)+A_h(\boldsymbol{p}-\boldsymbol{p}_h,\boldsymbol{w}_h)\right|}{\|\boldsymbol{w}_h\|}\\
&\leq C\|\boldsymbol{p}-\Pi_h^{IFE}\boldsymbol{p}\|+C\sup_{\boldsymbol{w}_h\in \mathbb{K}_h\backslash\{0\} }\frac{\left|A_h(\boldsymbol{p}-\boldsymbol{p}_h,\boldsymbol{w}_h)\right|}{\|\boldsymbol{w}_h\|}.
\end{aligned}
\end{equation}
The first term can be bounded by Lemma~\ref{lem_inter_newnorm}
\begin{equation}\label{pro_l202}
\|\boldsymbol{p}-\Pi_h^{IFE}\boldsymbol{p}\|\leq Ch\left(\|\boldsymbol{p}\|^2_{H^1(\Omega^+\cup\Omega^-)}+\eta\sum_{s=+,-}\|\boldsymbol{p}\|^2_{H^2\left(U (\Gamma, \delta_0)\cap \Omega^s\right) }\right)^{1/2}.
\end{equation}
We now consider the latter term. For all $\boldsymbol{w}_h\in\mathbb{K}_h$, we know from (\ref{def_kernel}) that  $b_h(\boldsymbol{w}_h,u_h)=b_h(\boldsymbol{w}_h,P^0_hu)=0$. Then,
 it follows from (\ref{inconsis}) that 
\begin{equation}\label{pro_l2_decom}
\begin{aligned}
\left|A_h(\boldsymbol{p}-\boldsymbol{p}_h,\boldsymbol{w}_h)\right|&\leq\left|b_h(\boldsymbol{w}_h,u-P_h^0u)\right|+\left|a_h(\boldsymbol{p},\boldsymbol{w}_h)-a(\boldsymbol{p},\boldsymbol{w}_h)\right|+\sum_{e\in\mathcal{E}_h^\Gamma}\left|\int_eu[\boldsymbol{w}_h\cdot\textbf{n}_{e}]_eds\right|.
\end{aligned}
\end{equation}
From (\ref{bound_b}) and  (\ref{equ_kh}), we can bound the first term on the right-hand side of (\ref{pro_l2_decom}) by 
\begin{equation}\label{pro_est001}
\left|b_h(\boldsymbol{w}_h,u-P_h^0u)\right|\leq \|\boldsymbol{w}_h\|_{L^2(\Omega)}\|u-P_h^0u\|_{L^2(\Omega)}\leq Ch|u|_{H^1(\Omega)}\|\boldsymbol{w}_h\|_{L^2(\Omega)}.
\end{equation}
For the second term on the right-hand side of (\ref{pro_l2_decom}),  recalling the definition of  $T^\triangle$ in (\ref{def_t_delta}) and using (\ref{pro_int6}),  (\ref{strip1}) and the extension result (\ref{extension}),  we have
\begin{equation}\label{pro_est002}
\begin{aligned}
&\left|a_h(\boldsymbol{p},\boldsymbol{w}_h)-a(\boldsymbol{p},\boldsymbol{w}_h)\right|=\left|\int_\Omega(\beta_h-\beta)\boldsymbol{p}\cdot\boldsymbol{w}_hd\boldsymbol{x}\right|\leq \sum_{T\in\mathcal{T}_h^\Gamma}\|\beta^+-\beta^-\|_{L^\infty(T^\triangle)}\int_{T^\triangle}\left|\boldsymbol{p}\cdot\boldsymbol{w}_h\right|d\boldsymbol{x}\\
&\qquad\qquad\qquad\leq C \|\boldsymbol{p}\|_{L^2(U(\Gamma,Ch_\Gamma^2))}\|\boldsymbol{w}_h\|_{L^2(\Omega)}\leq \sum_{s=+,-}C \|\boldsymbol{p}^s_E\|_{L^2(U(\Gamma,Ch_\Gamma^2))}\|\boldsymbol{w}_h\|_{L^2(\Omega)} \\
&\qquad\qquad\qquad\leq \sum_{s=+,-}C h_\Gamma\|\boldsymbol{p}^s_E\|_{H^1(U(\Gamma,\delta_0))}\|\boldsymbol{w}_h\|_{L^2(\Omega)} \leq Ch_\Gamma\|\boldsymbol{p}\|_{H^1(\Omega^+\cup\Omega^-)}\|\boldsymbol{w}_h\|_{L^2(\Omega)}.
\end{aligned}
\end{equation}
Note that $\int_e[\boldsymbol{w}_h\cdot\textbf{n}_{e}]_eds=0$ for all $e\in\mathcal{E}_h$, then the third term  
on the right-hand side of (\ref{pro_l2_decom}) can be written as 
\begin{equation}\label{pro_l2e_01}
\begin{aligned}
\sum_{e\in\mathcal{E}_h^\Gamma}\left|\int_eu[\boldsymbol{w}_h\cdot\textbf{n}_{e}]_eds\right|&=\sum_{e\in\mathcal{E}_h^\Gamma}\left|\int_e(u-c_e)[\boldsymbol{w}_h\cdot\textbf{n}_{e}]_eds\right|\\
&\leq \left(\sum_{e\in\mathcal{E}_h^\Gamma} \|u-c_e\|^2_{L^2(e)}\right)^{1/2}\left(\sum_{e\in\mathcal{E}_h^\Gamma} \|[\boldsymbol{w}_h\cdot\textbf{n}_{e}]_e\|^2_{L^2(e)}\right)^{1/2}\\
&\leq \left(\sum_{e\in\mathcal{E}_h^\Gamma} \|u-c_e\|^2_{L^2(e)}\right)^{1/2}\eta^{-1/2}\|\boldsymbol{w}_h\|,
\end{aligned}
\end{equation}
where $c_e$ is an arbitrary constant on the edge $e$.  By the standard trace inequality, Lemma~(\ref{lem_strip}) and the extension result (\ref{extension}), 
\begin{equation}\label{pro_l2e_02}
\begin{aligned}
&\sum_{e\in\mathcal{E}_h^\Gamma}\|u-c_e\|^2_{L^2(e)}\leq C\sum_{T\in\mathcal{T}_h^\Gamma}\left(h_T^{-1}\|u-P^0_Tu\|^2_{L^2(T)}+h_T|u|^2_{H^1(T)}\right)\leq Ch_\Gamma|u|^2_{H^1(U(\Gamma,h_\Gamma))}\\
&\qquad\leq Ch_\Gamma\sum_{s=+,-}|u_E^s|^2_{H^1(U(\Gamma,h_\Gamma))}\leq Ch_\Gamma^2\sum_{s=+,-}\|u_E^s\|^2_{H^2(U(\Gamma,\delta_0))}\leq Ch_\Gamma^2\|u\|^2_{H^2(\Omega^+\cup\Omega^-)}.
\end{aligned}
\end{equation}
%Similarly,
%\begin{equation}\label{pro_l2e_03}
%\begin{aligned}
%\|[\boldsymbol{w}_h\cdot\textbf{n}_{e}-c_e]_e\|^2_{L^2(e)}&\leq C\sum_{T\in\mathcal{T}_h^e}\left(h_T^{-1}\|\boldsymbol{w}_h\cdot\textbf{n}_{e}-P^0_T(\boldsymbol{w}_h\cdot\textbf{n}_{e})\|^2_{L^2(T)}+h_T|\boldsymbol{w}_h\cdot\textbf{n}_{e}|^2_{H^1(T)}\right)\\
%&\leq C\sum_{T\in\mathcal{T}_h^e}h_T|\boldsymbol{w}_h\cdot\textbf{n}_{e}|^2_{H^1(T)} \leq C\sum_{T\in\mathcal{T}_h^e}h_T|\boldsymbol{w}_h|^2_{H^1(T)}.
%\end{aligned}
%\end{equation}
%Combining (\ref{pro_l2e_01})-(\ref{pro_l2e_03}), we obtain the following estimate for the third term  
%on the right-hand side of (\ref{pro_l2_decom}) 
Therefore, we get the following error estimate on interface edges
\begin{equation}\label{pro_est003}
\sum_{e\in\mathcal{E}_h^\Gamma}\left|\int_eu[\boldsymbol{w}_h\cdot\textbf{n}_{e}]_eds\right|\leq C\eta^{-1/2}h_\Gamma\|u\|_{H^2(\Omega^+\cup\Omega^-)}\|\boldsymbol{w}_h\|.
\end{equation}
Substituting (\ref{pro_est001}), (\ref{pro_est002}) and (\ref{pro_est003}) into (\ref{pro_l2_decom}), we have
\begin{equation*}
\left|A_h(\boldsymbol{p}-\boldsymbol{p}_h,\boldsymbol{w}_h)\right| \leq Ch\left(\|\boldsymbol{p}\|_{H^1(\Omega^+\cup\Omega^-)}+(1+\eta^{-1/2})\|u\|_{H^2(\Omega^+\cup\Omega^-)}\right) \|\boldsymbol{w}_h\|
\end{equation*}
which together with (\ref{pro_l201}) and (\ref{pro_l202}) yields the desired estimate
\begin{equation}\label{pro_est1}
\|\boldsymbol{p}-\boldsymbol{p}_h\| \leq Ch\left(\|\boldsymbol{p}\|_{H^1(\Omega^+\cup\Omega^-)}+\eta^{1/2}\|\boldsymbol{p}\|_{H^2(\Omega_{\delta_0}^+\cup \Omega_{\delta_0}^- )}+(1+\eta^{-1/2})\|u\|_{H^2(\Omega^+\cup\Omega^-)}\right).
\end{equation}

Finally, we derive the error estimate for the solution $u_h$. It follows from the triangle inequality, the inf-sup condition~(\ref{lbbcon}) and the continuity (\ref{bound_b}) that 
\begin{equation}\label{ulti001}
\begin{aligned}
&\|u-u_h\|_{L^2(\Omega)}\leq \|u-P_h^0u\|_{L^2(\Omega)}+\|P_h^0u-u_h\|_{L^2(\Omega)}\\
&\qquad \leq \|u-P_h^0u\|_{L^2(\Omega)}+C\sup_{\boldsymbol{q}_h\in\mathcal{IRT}(\mathcal{T}_h)}\frac{\left|b_h(\boldsymbol{q}_h,P_h^0u-u_h)\right|}{\interleave \boldsymbol{q}_h\interleave}\\
&\qquad \leq \|u-P_h^0u\|_{L^2(\Omega)}+C\sup_{\boldsymbol{q}_h\in\mathcal{IRT}(\mathcal{T}_h)}\frac{\left|b_h(\boldsymbol{q}_h,P_h^0u-u)+b_h(\boldsymbol{q}_h,u-u_h)\right|}{\interleave \boldsymbol{q}_h\interleave }\\
&\qquad \leq C\|u-P_h^0u\|_{L^2(\Omega)}+C\sup_{\boldsymbol{q}_h\in\mathcal{IRT}(\mathcal{T}_h)}\frac{\left|b_h(\boldsymbol{q}_h,u-u_h)\right|}{\interleave \boldsymbol{q}_h\interleave }.
\end{aligned}
\end{equation}
The first term can be bounded easily 
\begin{equation}\label{ulti002}
\|u-P_h^0u\|_{L^2(\Omega)}\leq Ch|u|_{H^1(\Omega)}.
\end{equation}
For the second term, we know form (\ref{inconsis}) that  
\begin{equation*}
\left|b_h(\boldsymbol{q}_h,u-u_h)\right|\leq \left|A_h(\boldsymbol{p}-\boldsymbol{p}_h,\boldsymbol{q}_h)\right| +\left|a_h(\boldsymbol{p},\boldsymbol{q}_h)-a(\boldsymbol{p},\boldsymbol{q}_h)\right|+\sum_{e\in\mathcal{E}_h^\Gamma}\left|\int_eu[\boldsymbol{q}_h\cdot\textbf{n}_{e}]_eds\right|.
\end{equation*}
Similar to (\ref{pro_est002}) and (\ref{pro_est003}), it holds
\begin{equation}\label{ulti003}
\begin{aligned}
&\left|a_h(\boldsymbol{p},\boldsymbol{q}_h)-a(\boldsymbol{p},\boldsymbol{q}_h)\right|+\sum_{e\in\mathcal{E}_h^\Gamma}\left|\int_eu[\boldsymbol{q}_h\cdot\textbf{n}_{e}]_eds\right|\\
&\qquad \leq  Ch_\Gamma\|\boldsymbol{p}\|_{H^1(\Omega^+\cup\Omega^-)}\|\boldsymbol{q}_h\|_{L^2(\Omega)}+C\eta^{-1/2}h_\Gamma\|u\|_{H^2(\Omega^+\cup\Omega^-)}\|\boldsymbol{q}_h\|.
\end{aligned}
\end{equation}
On the other hand, (\ref{bound_a}) and (\ref{pro_est1}) imply 
\begin{equation}\label{ulti004}
\begin{aligned}
&\left|A_h(\boldsymbol{p}-\boldsymbol{p}_h,\boldsymbol{q}_h)\right|\leq C\|\boldsymbol{p}-\boldsymbol{p}_h\|\|\boldsymbol{q}_h\| \\
&\qquad\leq Ch\left(\|\boldsymbol{p}\|_{H^1(\Omega^+\cup\Omega^-)}+\eta^{1/2}\|\boldsymbol{p}\|_{H^2(\Omega_{\delta_0}^+\cup \Omega_{\delta_0}^- )}+(1+\eta^{-1/2})\|u\|_{H^2(\Omega^+\cup\Omega^-)}\right)\|\boldsymbol{q}_h\|.
\end{aligned}
\end{equation}
Combining (\ref{ulti001})-(\ref{ulti004}) and using the fact $ \| \boldsymbol{q}_h \|\leq \interleave \boldsymbol{q}_h \interleave $, we get the desired estimate
\begin{equation*}
\|u-u_h\|_{L^2(\Omega)}\leq Ch\left(\|\boldsymbol{p}\|_{H^1(\Omega^+\cup\Omega^-)}+\eta^{1/2}\|\boldsymbol{p}\|_{H^2(\Omega_{\delta_0}^+\cup \Omega_{\delta_0}^- )}+(1+\eta^{-1/2})\|u\|_{H^2(\Omega^+\cup\Omega^-)}\right).
\end{equation*}
\end{proof}
\begin{remark}
From Theorem~\ref{theo_error}, we know the penalty $s_h(\cdot,\cdot)$ with $\eta>0$ is necessary to ensure the optimal convergence rates which is confirmed in the next section, although the IFE method is stable without the penalty (i.e., $\eta=0$) in view of Lemma \ref{lem_coercivity_ah} and Lemma~\ref{lem_infsup}. Throughout the proof, we find the issue is caused by the inequality (\ref{pro_l2e_01}) which does not hold if $\eta=0$.
\end{remark}

\section{Numerical examples}\label{sec_num}
 In this section, we present some numerical examples to validate the theoretical analysis. We also compare the proposed immersed Raviart-Thomas mixed finite element method (Immersed RT) with the traditional Raviart-Thomas mixed finite element method (Traditional RT) which reads: find $(\boldsymbol{p}_h,u_h)\in \mathcal{RT}(\mathcal{T}_h)\times M_h$ such that
\begin{equation*}
\begin{aligned}
a_h(\boldsymbol{p}_h,\boldsymbol{q}_h)+b_h(\boldsymbol{q}_h,u_h)&=0 \qquad &&\forall \boldsymbol{q}_h\in \mathcal{RT}(\mathcal{T}_h),\\
b_h(\boldsymbol{p}_h,v_h)&=F(v_h)\qquad &&\forall v_h\in M_h,
\end{aligned}
\end{equation*}
where $\mathcal{RT}(\mathcal{T}_h)$ is the standard Raviart-Thomas space defined in (\ref{rt_space}).
For simplicity, we take $\Omega=(-1,1)\times(-1,1)$ as the computational domain and use uniform triangulations constructed as follows. We first partition the domain into $N\times N$ congruent rectangles, and then obtain the triangulation by cutting the rectangles along one of diagonals in the same direction. We only report the errors $\|\boldsymbol{p}-\boldsymbol{p}_h\|_{L^2}$ and the corresponding convergence rates.  The numerical results for $u_h$ are not listed  because they are almost the same and the corresponding convergence rates are  $O(h)$  for different methods.
% From the numerical experiments, we find that the errors $\|u-u_h\|_{L^2}$ and the corresponding convergence rates are almost the same for different methods. So we only report numerical results for $\boldsymbol{p}_h$.

 \textbf{Example 1}. We first consider a benchmark example from \cite{Li2003new} which has been used in many articles.
The interface is $\Gamma=\{(x_1,x_2)\in \mathbb{R}^2: x_1^2+x_2^2=r_0^2\}$ which separates $\Omega$ into $\Omega^-=\{(x_1,x_2)\in \mathbb{R}^2: x_1^2+x_2^2<r_0^2\}$ and $\Omega^+=\{(x_1,x_2)\in \Omega: x_1^2+x_2^2>r_0^2\}$.
The exact solution to the interface problem is chosen as
\begin{equation*}
u(\boldsymbol{x})=\left\{
\begin{aligned}
&(x_1^2+x_2^2)^{3/2}/\widetilde{\beta}^-\qquad &\mbox{ in }\Omega^-,\\
&(x_1^2+x_2^2)^{3/2}/\widetilde{\beta}^++\left(1/\widetilde{\beta}^--1/\widetilde{\beta}^+\right)r_0^3 &\mbox{ in }\Omega^+,
\end{aligned}
\right.
\end{equation*}
where $r_0=0.5$, $\widetilde{\beta}^+=10^{-2}$ and $\widetilde\beta^-=1$.

%\begin{table}[H]
%\caption{The $\|u-u_h\|_{L^2}$ errors and  convergence rates for \textbf{Example 1}.\label{ex1_u}}
%\begin{center}
%\begin{tabular}{|c|c c|c c|c c|}
%  \hline
%  & \multicolumn{2}{c|}{RT-FEM} & \multicolumn{2}{c|}{IRT-FEM ($\eta=0$)}& \multicolumn{2}{c|}{ IRT-FEM  ($\eta=1$)}\\ \hline
%       $N$  &  $\|u-u_h\|_{L^2}$   &  rate   &  $\|u-u_h\|_{L^2}$  &  rate  &  $\|u-u_h\|_{L^2}$  &  rate     \\ \hline
%\end{tabular}
%\end{center}
%\end{table}

\begin{table}[H]
\caption{The $\|\boldsymbol{p}-\boldsymbol{p}_h\|_{L^2}$ errors and  convergence rates   for \textbf{Example 1}.\label{ex1_p}}
\begin{center}
\begin{tabular}{|c|c c|c c|c c|}
  \hline
  & \multicolumn{2}{c|}{Traditional RT} & \multicolumn{2}{c|}{Immersed RT ($\eta=0$)}& \multicolumn{2}{c|}{ Immersed RT  ($\eta=1$)}\\ \hline
       $N$  &  $\|\boldsymbol{p}-\boldsymbol{p}_h\|_{L^2}$   &  rate   &  $\|\boldsymbol{p}-\boldsymbol{p}_h\|_{L^2}$ &  rate  &  $\|\boldsymbol{p}-\boldsymbol{p}_h\|_{L^2}$  &  rate     \\ \hline
      8  & 3.033E-01 &             & 3.770E-01  &           &  3.649E-01  &                \\ \hline
    16  & 1.477E-01  &    1.04 & 3.068E-01  &    0.30 & 2.763E-01  &    0.40       \\ \hline
    32  & 7.322E-02  &    1.01 &  2.980E-01 &    0.04 & 2.002E-01  &    0.46      \\ \hline
    64  & 3.653E-02  &    1.00 & 1.772E-01  &    0.75 & 9.581E-02   &   1.06       \\ \hline
  128  & 1.824E-02  &    1.00 & 1.546E-01  &    0.20 & 6.386E-02   &   0.59       \\ \hline
  256  & 9.115E-03   &   1.00 & 1.105E-01   &   0.48 & 3.117E-02    &  1.03        \\ \hline
  512  & 4.554E-03   &   1.00 & 7.855E-02  &    0.49 & 1.502E-02    &  1.05       \\ \hline
\end{tabular}
\end{center}
\end{table}

The numerical results are presented in Table~\ref{ex1_p}.  We observe the optimal convergence rates for the immersed  Raviart-Thomas mixed finite element method with $\eta=1$ and suboptimal convergence rates for the method without penalty, i.e., $\eta=0$, which are in good agreement with Theorem~\ref{theo_error}.
 
From  the second column of Table~\ref{ex1_p}, it is surprising that the solution of the traditional Raviart-Thomas mixed finite element method also converges optimally even if unfitted meshes are used for this interface problem. Below we explain what is happening.
We find that the exact solution of this example is a constant along the interface $\Gamma$. Thus, the tangential derivative of the exact solution is zero on the interface, i.e.,  $(\nabla u\cdot {\rm\mathbf{t}})|_{\Gamma}=0$, which implies a special interface condition: $[\boldsymbol{p}\cdot \mathbf{t}]_\Gamma=[\widetilde{\beta}\nabla u\cdot \mathbf{t}]_\Gamma=0$. On the other hand, the standard Raviart-Thomas functions also satisfy this interface condition, i.e., $[\boldsymbol{p}_h\cdot \mathbf{t}]_\Gamma=0$ for all  $\boldsymbol{p}_h\in\mathcal{RT}(\mathcal{T}_h)$. We have test many other numerical examples from the literature satisfying $(\nabla u\cdot {\rm\mathbf{t}})|_{\Gamma}=0$ and similar optimal convergence rates have also been observed.
 
 \textbf{Example 2}.  In order to show the suboptimal convergence of the traditional Raviart-Thomas mixed finite element method, we test an example with  $(\nabla u\cdot {\rm\mathbf{t}})|_{\Gamma}\not=0$ which was constructed in \cite{2021ji_nonconform}.  The interface is $\Gamma=\{(x_1,x_2)\in \mathbb{R}^2: x_1^2+x_2^2=r_0^2\}$ which separates $\Omega$ into $\Omega^-=\{(x_1,x_2)\in \mathbb{R}^2: x_1^2+x_2^2<r_0^2\}$ and $\Omega^+=\{(x_1,x_2)\in \Omega: x_1^2+x_2^2>r_0^2\}$. Let $(r,\theta)$ be the polar coordinate of  $\boldsymbol{x}=(x_1,x_2)$. The exact solution is chosen as $u(\boldsymbol{x})=j(\boldsymbol{x})v(\boldsymbol{x})\sin(\theta)$, where
\begin{equation*}
j(\boldsymbol{x})=\left\{
\begin{aligned}
&\exp\left(-\frac{1}{1-(r-r_0)^2/\eta^2}\right)~&\mbox{ if } |r-r_0|<\eta,\\
&0 &\mbox{ if } |r-r_0|\geq\eta,
\end{aligned}\right.
\end{equation*}
and
\begin{equation*}
v(\boldsymbol{x})=\left\{
\begin{aligned}
&1+(r^2-r_0^2)/\widetilde{\beta}^+~&\mbox{ if } \boldsymbol{x}\in\Omega^+,\\
&1+(r^2-r_0^2)/\widetilde{\beta}^-~&\mbox{ if } \boldsymbol{x}\in\Omega^-.
\end{aligned}\right.
\end{equation*}
Let $r_0=0.5$, $\eta=0.45$, $\widetilde{\beta}^+=10^{-2}$ and $\widetilde\beta^-=1$. It is easy to verify that the jump conditions (\ref{p1.2})-(\ref{p1.3}) are satisfied and $\nabla u\cdot {\rm\mathbf{t}}\not=0$ on $\Gamma$.

%\begin{table}[H]
%\caption{The $\|u-u_h\|_{L^2}$  errors and convergence rates for \textbf{Example 2}.\label{ex2_u}}
%\begin{center}
%\begin{tabular}{|c|c c|c c|c c|}
%  \hline
%  & \multicolumn{2}{c|}{RT-FEM} & \multicolumn{2}{c|}{IRT-FEM ($\eta=0$)}& \multicolumn{2}{c|}{ IRT-FEM  ($\eta=1$)}\\ \hline
%       $N$  &  $\|u-u_h\|_{L^2}$   &  rate   &  $\|u-u_h\|_{L^2}$  &  rate  &  $\|u-u_h\|_{L^2}$  &  rate     \\ \hline
%      8 & 5.144E+00 &             & 5.152E+00   &            & 5.151E+00   &        \\ \hline
%    16 & 2.632E+00  &    0.97 & 2.524E+00   &   1.03 &  2.532E+00   &   1.02     \\ \hline
%    32 & 6.293E-01   &   2.06 &  6.443E-01   &   1.97 &   6.393E-01   &   1.99     \\ \hline
%    64 & 2.984E-01    &  1.08  & 2.988E-01   &   1.11 &   2.978E-01    &  1.10      \\ \hline
%  128 & 1.491E-01    &  1.00  & 1.498E-01   &   1.00 &   1.490E-01   &   1.00       \\ \hline
%  256 & 7.458E-02    &  1.00  & 7.500E-02  &    1.00 &   7.454E-02   &   1.00      \\ \hline
%  512 & 3.730E-02    &  1.00  & 3.752E-02 &     1.00 &   3.728E-02  &    1.00      \\ \hline
%\end{tabular}
%\end{center}
%\end{table}

\begin{table}[H]
\caption{The $\|\boldsymbol{p}-\boldsymbol{p}_h\|_{L^2}$ errors and  convergence rates   for \textbf{Example 2}.\label{ex2_p}}
\begin{center}
\begin{tabular}{|c|c c|c c|c c|}
  \hline
  & \multicolumn{2}{c|}{Traditional RT} & \multicolumn{2}{c|}{Immersed RT ($\eta=0$)}& \multicolumn{2}{c|}{ Immersed RT  ($\eta=1$)}\\ \hline
       $N$  &  $\|\boldsymbol{p}-\boldsymbol{p}_h\|_{L^2}$   &  rate   &  $\|\boldsymbol{p}-\boldsymbol{p}_h\|_{L^2}$ &  rate  &  $\|\boldsymbol{p}-\boldsymbol{p}_h\|_{L^2}$  &  rate     \\ \hline
      8  &  6.946E-01 &           & 6.595E-01    &            & 6.563E-01   &              \\ \hline
    16  & 4.107E-01  &    0.76 & 3.069E-01  &    1.10  & 3.069E-01   &   1.10       \\ \hline
    32  & 2.100E-01   &   0.97 & 1.652E-01  &    0.94  & 1.438E-01   &   1.09      \\ \hline
    64  & 1.376E-01   &   0.61 & 8.953E-02  &    0.88  & 7.081E-02   &   1.02       \\ \hline
  128  & 8.944E-02   &   0.62 & 6.300E-02  &    0.51  & 3.817E-02   &   0.89       \\ \hline
  256  & 6.056E-02   &   0.56 & 4.180E-02  &    0.59  & 1.888E-02   &   1.02       \\ \hline
  512  & 4.190E-02   &   0.53 & 2.851E-02  &    0.55  & 9.331E-03  &    1.02       \\ \hline
\end{tabular}
\end{center}
\end{table}

Numerical results presented in Table~\ref{ex2_p} clearly show the optimal convergence rates for the immersed Raviart-Thomas mixed finite element method with $\eta=1$ and suboptimal convergence rates for other methods.

% \textbf{Example 3}. Finally, we consider an example with variable coefficients and and a non-convex interface. The interface is the zero level set of the function
% $$\varphi(\boldsymbol{x})=(3(x_1^2+x_2^2)-x_1)^2-x_1^2-x_2^2+0.02$$
% and the exact solution is chosen as $u(\boldsymbol{x})=\varphi(\boldsymbol{x})/\widetilde{\beta}(\boldsymbol{x})$, where
% \begin{equation*}
% \widetilde{\beta}(\boldsymbol{x})=\left\{
% \begin{aligned}
% &\widetilde{\beta}^+(\boldsymbol{x})=300(2+\sin(6x_1+6x_2))& \mbox{ if } \varphi(\boldsymbol{x})>0,\\
% &\widetilde{\beta}^-(\boldsymbol{x})=2+\cos(6x_1+6x_2) &\mbox{ if } \varphi(\boldsymbol{x})<0.
% \end{aligned}\right.
% \end{equation*}
%This example is taken from \cite{}. It is easy to verify that the jump conditions (\ref{p1.2})-(\ref{p1.3}) are satisfied.  For this problem with variable coefficients, in order to satisfy the condition (\ref{cond_beta}), we choose $\beta_T^s=\beta^s(\boldsymbol{x}_T)$, $s=+,-$,  where $\boldsymbol{x}_T$ is the midpoint of $\Gamma_h\cap T$. 

 \section{Conclusions}\label{sec_con}
In this paper, we have constructed an immersed  Raviart-Thomas finite element and proposed a corresponding mixed finite element method for solving second-order elliptic interface problems on unfitted meshes. We have shown that the immersed  Raviart-Thomas finite element space is nonconforming and a penalty term on the interface edges is necessary to ensure the optimal convergence.  Some important properties of the immersed  Raviart-Thomas finite element space have also been derived including the unisolvence of IFE basis functions, the optimal approximation capabilities of the IFE space and the corresponding commuting digram.  Moreover, we have proved the inf-sup condition of the proposed IFE method and derived the optimal error estimates which are confirmed by numerical examples.

\bibliographystyle{plain}

\end{document}